\documentclass{IEEEtran}
\usepackage{cite}
\usepackage{amsmath,amssymb,amsfonts}
\usepackage{amsthm}
\usepackage{color}
\usepackage{algorithmicx,algorithm}
\usepackage{graphicx}
\usepackage{textcomp}
\usepackage[noend]{algpseudocode}
\usepackage{bbm}
\usepackage{bm}
\usepackage{booktabs}
\newtheorem{definition}{Definition}[section]
\newtheorem{theorem}{Theorem}[section]
\newtheorem{lemma}{Lemma}[section]

\newtheorem{remark}{Remark}[section]
\newtheorem{proposition}{Proposition}[section]
\newtheorem{assumption}{Assumption}[section]
\def\BibTeX{{\rm B\kern-.05em{\sc i\kern-.025em b}\kern-.08em
T\kern-.1667em\lower.7ex\hbox{E}\kern-.125emX}}
\begin{document}
\title{Distributed Coverage Control on Poriferous Surface via Poly-Annulus Conformal Mapping}
\author{Xun Feng, Chao Zhai, \IEEEmembership{Senior Member, IEEE}

\thanks{This work is supported by the ``CUG Scholar" Scientific Research Funds at China University of Geosciences (Wuhan) (Project No. 2020138).}
\thanks{Xun Feng and Chao Zhai are with School of Automation, China University of Geosciences, Wuhan 430074, China, and with Hubei Key Laboratory of Advanced Control and Intelligent Automation for Complex Systems and also with Engineering Research Center of Intelligent Technology for Geo-Exploration, Ministry of Education (Corresponding author: Chao Zhai, email: zhaichao@amss.ac.cn).}}

\maketitle

\begin{abstract}
The inherent non-convexity of poriferous surfaces typically entraps agents in local minima and complicates workload distribution. To resolve this, we propose a distributed diffeomorphic coverage control framework for the multi-agent system (MAS) in such surfaces. First, we establish a distributed poly-annulus conformal mapping that transforms arbitrary poriferous surfaces into a multi-hole disk. Leveraging this topological equivalence, a collision-free sectorial partition mechanism is designed in the multi-hole disk, which rigorously induces strictly connected subregions and workload balance on the poriferous surfaces. This mechanism utilizes a buffer-based sequence mechanism to ensure strict topological safety when bypassing obstacles. Furthermore, a pull-back Riemannian metric is constructed to define the length metric that encodes safety constraints. Based on this metric, a distributed gradient-based control law is synthesized to drive agents toward optimal configurations, ensuring simultaneous obstacle avoidance and coverage optimization. Theoretical analyses guarantee the Input-to-State Stability (ISS) of the partition dynamics and the asymptotic convergence of the closed-loop system. Numerical simulations confirm the reachability and robustness of the proposed coverage algorithm, offering a scalable solution for distributed coverage in poriferous surfaces.
\end{abstract}


\section{Introduction}
\label{Introduction}
The rapid increase of low-cost smart agents has catalyzed significant advancements in multi-agent coordination, enabling complex missions such as environmental monitoring, cooperative search, and precision agriculture~\cite{Cao11, vehicle23}. Among these, coverage control stands out as a fundamental problem, aiming to optimally deploy a sensor network to monitor a field of interest with respect to a specific density function~\cite{sun19}. Graph theoretic methods~\cite{mesbahi} and cooperative control~\cite{zhai21} have provided solid foundations for these tasks.

Early research extensively explores coverage control in convex domains by utilizing Voronoi tessellation to partition the region and deploying Lloyd's algorithm to drive agents to the centroids~\cite{Cortes04, Pimenta08}. However, standard Euclidean metrics often leads to local minima in non-convex domains. To address non-convexity, Bhattacharya et al.~\cite{Unknown13} employs grid-based wavefront propagation to approximate geodesic Voronoi partitions. Nevertheless, the grid resolution limits coverage precision, thereby preventing the generation of continuous gradient flows on smooth manifolds. Although Bhattacharya et al~\cite{SB14} explore coverage on Riemannian manifolds, their approach incurs prohibitive computational costs and lacks mechanisms for workload balance. In contrast, geodesic Voronoi partitions are proposed to achieve equitable workload distribution~\cite{Cortes10}. Similarly, Palacios-Gasós et al.~\cite{equitable19} integrate power diagrams with graph-based planning to generate sweep paths. Yet, this method requires complex geometric approximations. Sectorial partitions~\cite{zhai23} and hierarchical protocols~\cite{porous} offer alternative strategies but impose restrictive constraints on obstacle shapes. More recently, Surendhar et al.~\cite{obstacles} integrate adaptive estimation with artificial potential fields (APF) for obstacle avoidance. Such APF-based methods are prone to trapping agents in local minima in poriferous surfaces.

To simplify the environmental topology, conformal geometry~\cite{Gu08} offers a rigorous mathematical pathway. Caicedo-Nunez et al.~\cite{perform08} construct a diffeomorphism that maps physical obstacles to single points in a punctured convex region. Unfortunately, mapping solid obstacles to points induces singularities at the boundaries. Lekien and Leonard~\cite{Nonuniform08} utilize density-equalizing cartograms, but their diffusion-based approach is primarily limited to simply connected domains. While harmonic maps can transform workspaces into punctured disks~\cite{HM18}, the solution to Laplace equation relies on global boundary conditions, resulting in a centralized computation process. Fan et al.~\cite{CM23} develop conformal navigation transformations to map obstacles into a sphere world. However, the navigation functions stabilize agents at single goal points, rendering them unsuitable for the workload balance and centroid optimization. Vlantis et al.~\cite{obstacle23} combine harmonic maps with cell decomposition, yet this requires explicit space discretization, leading to discontinuous control laws or complex switching logic between partition cells. Similarly, the neighborhood-augmented graph search~\cite{SB23} focuses on finding distinct paths for single robots, lacking the capacity for multi-agent workload balance. Xu et al.~\cite{Xu20} propose time-varying diffeomorphisms for deforming domains, but it is restricted to simply connected regions and relies on centralized implementation. While combining quasi-conformal mappings with control barrier functions~\cite{Choi25} enables reactive navigation, these methods rely on local safety constraints. Consequently, it lacks the global gradient guidance required to the optimal coverage configuration. Above all, existing methods fail to provide a distributed solution for optimal coverage with workload balance on poriferous surfaces.

To overcome these limitations, this paper proposes a fully distributed diffeomorphic coverage control framework. Inspired by the centralized multiply-connected mapping~\cite{multiconnected21}, we construct a poly-annulus conformal mapping that transforms a poriferous surface into a multi-hole disk. This transformation allows to generalize the efficient sectorial partition to poriferous surfaces. Unlike traditional methods, our approach incorporates a buffer-based sequence mechanism to ensure strict topological safety. Furthermore, by utilizing partial welding techniques~\cite{welding} and linear Beltrami solvers~\cite{rectangular13}, we achieve this in a distributed manner. The Riemannian geometry principles~\cite{Riemannian geometry} and classical Koebe's iteration~\cite{Koebe} are leveraged to design a length metric that guarantees convergence. 
In brief, the core contributions of this work are threefold:
\begin{enumerate}
\item Construct a poly-annulus conformal mapping to regularize arbitrary poriferous surface into the multi-holed disk in a distributed manner.
\item Establish a modified partition mechanism to ensure strict topological safety and workload balance, which guarantees input-to-state stability of partition dynamics against geometric perturbations.
\item Synthesize a distributed control law that guarantees both collision-free navigation and asymptotic convergence to optimal coverage centroids by leveraging the pull-back length metric that encodes safety constraints.
\end{enumerate}

The remainder of this paper is organized as follows. Section~\ref{section2} mathematically formulates the coverage control problem on poriferous surface. Section~\ref{section3} details the distributed construction of poly-annulus conformal mapping. Section~\ref{section4} presents the main theoretical results, including the safe partition algorithm, Riemannian metric design, and stability analysis. Numerical simulations are provided in Section~\ref{section5}, and Section~\ref{section6} concludes the paper. 

\begin{figure}[t!]
\centering
\includegraphics[width=0.45\textwidth]{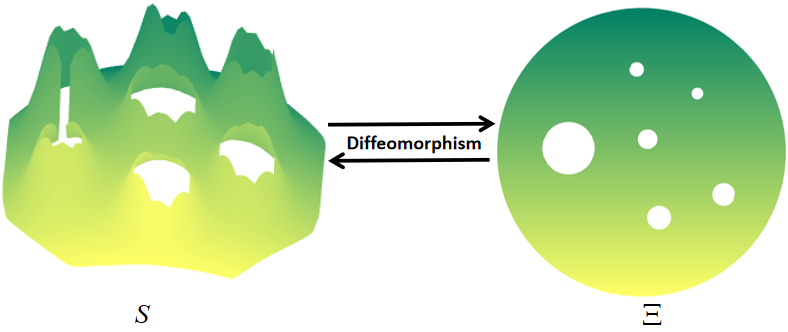}
\caption{\label{Fig.1} Topological relationship between the surface \(S\) and $n$-holed disk \(\Xi\). The color gradient illustrates the preservation of density distribution $\rho$.} \label{Fig.1.}
\end{figure}

\section{Problem Formulation}\label{section2}

This section formulates the coverage control problem of MAS on poriferous surface with boundary.
\subsection{Mathematical Description of Coverage Region}

Let $S$ be a compact and orientable 2-dimensional Riemannian manifold with boundary embedded in $\mathbb{R}^3$, which can be topologically viewed as a poly-annulus (see Fig.~\ref{Fig.1.}). The manifold $S$ represents the feasible workspace constrained by a set of $n$ disjoint obstacles, denoted as $\mathcal{O} = \{O_1, \dots, O_n\}$. Coordinating multi-agent coverage directly on $S$ is intractable, due to the metric distortion induced by the irregular geometries of obstacle boundaries $\partial O_k$. As demonstrated in \cite{Choi25}, the problem can be simplified by mapping the robot workspace to the ball world, specifically an $n$-holed disk $\Xi$ (i.e., a unit disk with radius $R=1$, perforated by $n$ disjoint circular obstacles). Since $S$ is homeomorphic to $\Xi$, we construct a diffeomorphism $\tau: S \to \Xi$ to convert the coverage region into the ball world. As $S$ is compact and diffeomorphism \(\tau\) preserves compactness, \(\Xi\) is also compact. We define the polar coordinates in the ball world as $(\hat{r}, \hat{\theta})$ to distinguish them from those in the original robot workspace. The $k$-th obstacle within the ball world $\Xi$ is uniquely characterized by its center $\hat{o}_k$ and radius $\hat{r}_k$. Consequently, the feasible region of the ball world is a $n$-holed disk, which is defined in vector form as: 
\begin{equation} \label{equation1}
\Xi = \left\{ \hat{p} \in \mathbb{R}^2 \middle| \|\hat{p}\| \le 1, \text{ s.t. } \|\hat{p} - \hat{o}_k\| > \hat{r}_k, \forall k \in \mathbb{I}_n \right\},
\end{equation}
where $\hat{p} = (\hat{r}, \hat{\theta})$ denotes the coordinate of agent in $\Xi$ and \(\mathbb{I}_N=\{1,2,...,N\}\). The inequality constraints ensure that the workspace is confined within \(\Xi\) while remaining disjoint from all internal circular obstacles.

Consider a network of $N$ agents deployed within the robot workspace $S$. The spatial distribution of workload is modeled by a continuous density function $\rho: S \to \mathbb{R}_{\ge 0}$. To achieve workload balance in \(\Xi\), we transport the workload distribution to $\Xi$ via the diffeomorphism $\tau$. Specifically, if $\tau$ scales lengths by a conformal factor $\lambda$, the density in $\Xi$ is given by
\begin{equation}
\hat{\rho}(\boldsymbol{\hat\xi}) = \rho(\tau^{-1}(\boldsymbol{\hat\xi})) \cdot \det(J_{\tau^{-1}}(\boldsymbol{\hat\xi})), \quad \forall \boldsymbol{\hat\xi} \in \Xi,
\end{equation}
where $J_{\tau^{-1}}(\hat{\xi})$ denotes the Jacobian of \(\tau^{-1}\). Note that the density within the obstacle regions is defined as zero. The sectorial partition can be implemented to divide the region \(\Xi\) and balance the workload of subregions~\cite{porous}. Each agent is assigned with a virtual partition bar, parameterized by an angular phase $\hat{\psi}_i \in [0, 2\pi)$. And the partition phases are numbered sequentially as $0 \le \hat{\psi}_1 < \hat{\psi}_2 < \dots < \hat{\psi}_N < 2\pi$. Consequently, the ball world is decomposed into $N$ sectors $\Xi = \bigcup_{i=1}^N \hat{E}_i$, where $\hat{E}_i$ is the region bounded by the angles $[\hat{\psi}_i, \hat{\psi}_{i+1})$. The dynamics of the $i$-th agent is presented as
\begin{equation}\label{dp}
\dot{\mathbf{p}}_i = \mathbf{u}_i
\end{equation}
with the control input $\mathbf{u}_i \in T_{p_i}S$, \(\forall i\in \mathbb{I}_N=
\{1,2,...,N\}\). While agents operate in the robot workspace $S$, the partition dynamics is computed in the ball world $\Xi$. Concretely, the workload $\hat{m}_i$ in \(\Xi\) is calculated by
\begin{equation}\label{mi}
\hat{m}_i = \int_{\hat{\psi}_i}^{\hat{\psi}_{i+1}} \hat{\omega}(\hat{\theta}) d\hat{\theta}, \quad \hat{\omega}(\hat{\theta}) = \int_{0}^{1} \hat{\rho}(\hat{r}, \hat{\theta}) \hat{r} d\hat{r},
\end{equation}
where $\hat{\omega}(\hat{\theta})$ represents the marginal density along the angular coordinate in $\Xi$. In addition, the radial integration is performed over $[0, 1]$, as the density $\hat{\rho}$ vanishes within the internal obstacles. To balance the workload among subregions, the dynamic of the \(i\)-th partition bar in \(\Xi\) is designed as
\begin{equation}\label{partitioning dynamic} 
\dot{\hat{\psi}}_i = k_{\hat{\psi}} (\hat{m}_i - \hat{m}_{i-1}), \quad \forall i \in \mathbb{I}_N, \end{equation}
where $k_{\hat{\psi}}$ is a positive constant, and \({\hat{\psi} _i}=\bmod ( {{\hat{\psi} _i},2\pi })\) is employed to ensure \({\hat{\psi} _i} \in [0,2\pi )\) with periodic boundary conditions $\hat{\psi}_{N+1} = \hat{\psi}_1 + 2\pi$ and $\hat{m}_0 = \hat{m}_N$. 

\subsection{Optimization Formulation for Multi-Agent Coverage}
To rigorously quantify the coverage performance, we define a global cost function directly in the robot workspace $S$. Inspired by the distortion energy framework in~\cite{Cortes04}, the performance index is formulated 
as
\begin{equation}\label{cost}
J(\hat{\boldsymbol{\psi}}, \mathbf{p}) = \sum\limits_{i = 1}^N {\int_{E_i(\hat{\boldsymbol{\psi}})} {f_i\left( {d_l(p_i, q)} \right) \rho(q) dq} },
\end{equation}
where $\mathbf{p} = (p_1,\dots,p_N)^T$ denotes the position vector of agents in \(S\), and $\hat{\boldsymbol{\psi}} = (\hat{\psi}_1, \dots, \hat{\psi}_N)^T $ represents the vector of partition phase angles in ball world \(\Xi\), which determines the workspace subregions via the inverse mapping $E_i(\hat{\boldsymbol{\psi}}) = \tau^{-1}(\hat{E}_i(\hat{\boldsymbol{\psi}}))$. In addition, $d_l(p_i, q)$ refers to the length metric designed in Subsection~\ref{Design of Length metric} to encode obstacle avoidance. To ensure the well-posedness of the performance index~\eqref{cost}, $f_i(\cdot)$ satisfies the following assumption~\cite{Cortes10}. 
\begin{assumption}[Admissible Performance Functions] \label{assf}
We require $f_i(\cdot)$ to satisfy:
\begin{enumerate}
    \item $f_i \in C^2([0, \infty))$.
    \item $f_i'(x) > 0$ and $f_i''(x) > 0$ for all $x \ge 0$. 
    \item $f_i(0)=0$ and $\lim_{x\to \infty} f_i(x) = \infty$. 
\end{enumerate}
A standard example satisfying these criteria is $f(x)=x^2$.
\end{assumption}
Our objective is to optimize coverage performance with workload balance over the robot workspace $S$. Since the diffeomorphism $\tau$ establishes a topological equivalence between the two domains, solving the coverage problem in the robot workspace $S$ is equivalent to solving its counterpart in the ball world $\Xi$. (see Theorem~\ref{the3}). Specifically, we design a distributed control law for the agent dynamics~\eqref{dp} in $S$ that utilizes the partition dynamic \eqref{partitioning dynamic} in $\Xi$ to drive agents toward their optimal positions.  Inspired by the coverage framework in~\cite{Cortes10}, the optimization problem is formulated as follows:
\begin{equation}\label{minJ}
\begin{split}
& \quad \quad \quad \min_{\hat{\mathbf{\psi}}, \mathbf{p}} J(\hat{\mathbf{\psi}}, \mathbf{p}) \\
s.t. \quad & |\hat{m}_i - \hat{m}_j| \le \delta(\beta), \quad \forall i,j \in \mathbb{I}_N \
\end{split}
\end{equation}
where $\delta(\beta)$ represents the workload deviation arising from topological safety constraints. The primary challenge of Problem \eqref{minJ} stems from the non-convexity and complex topology of the workspace. Our framework addresses these challenges by regularizing the workspace topology within the ball world while utilizing the conformal mapping and the pull-back length metric to ensure physical fidelity.
\begin{figure*}[t!]
\centering
\includegraphics[width=0.9\textwidth]{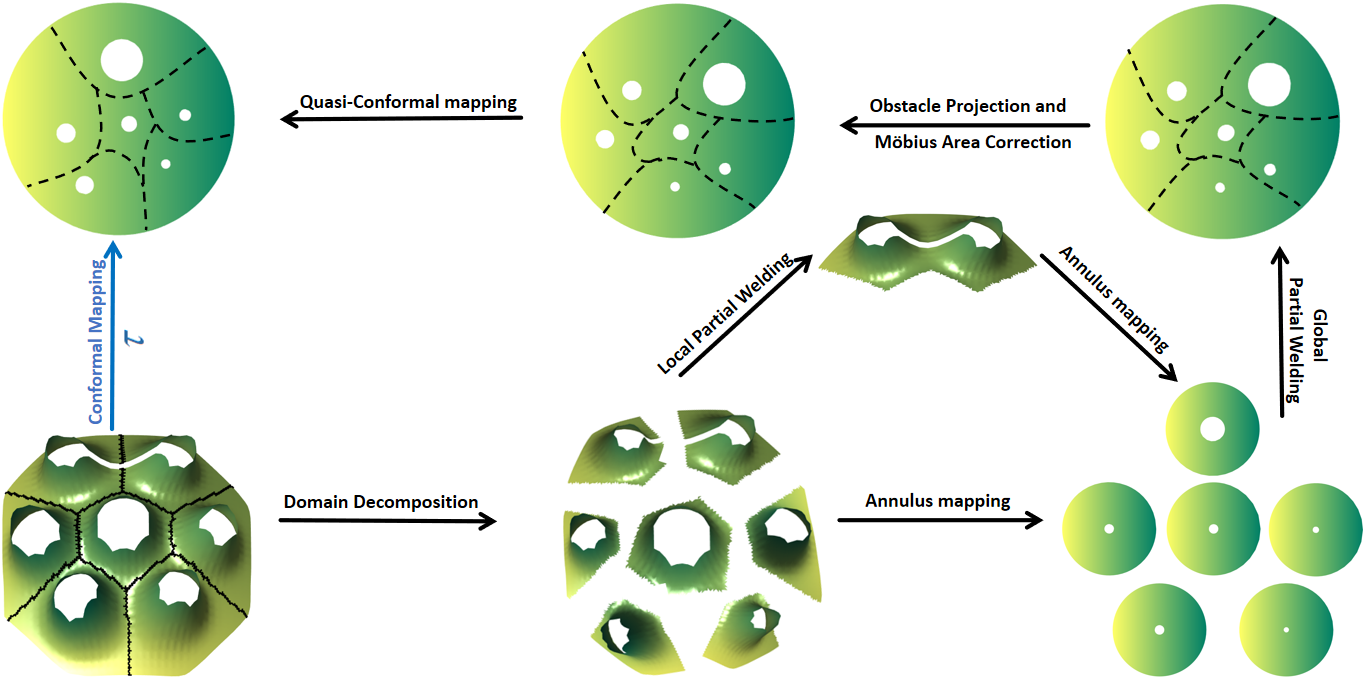}
\caption{\label{png: MAPProcess1} The workflow of Algorithm \ref{Algorithm3}. The process decomposes the complex topology into local sub-problems, stitches them via partial and global welding, and applies geometric rectification to generate a global coordinate system for the $n$-holed disk \(\Xi\).}
\label{Fig.2}
\end{figure*}

\section{Construction of Conformal Mapping}\label{section3}
This section details the construction of poly-annulus conformal mapping $\tau: S \to \Xi$. We employ a partition operator $\mathcal{P}$ to decouple $S$ into subdomains, reducing the global mapping into parallelizable local sub-problems. These local maps are integrated via partial welding and geometric rectification to yield a $n$-holed disk suitable for sectorial partition.

\subsection{Domain Decomposition} \label{sectionA}
To facilitate the distributed construction of poly-annulus conformal mapping, we establish the existence of a partition operator $\mathcal{P}$ to reduce the topological complexity of $S$. 
\begin{definition}[Partition Operator Admissibility]
A partition operator $\mathcal{P}$ that decomposes $S$ into $M$ disjoint subdomains $\{S_k\}_{k=1}^M$ ($M\ge n$) is said to be admissible if every $S_k$ is homeomorphic to either a unit disk $\mathbb{D}$ or a unit annulus $\mathbb{A}$.
\end{definition}
Based on the Triangulation Theorem and topological surgery~\cite{Gu08}, $S$ can be segmented into a finite set of simply connected patches and annular patches. This ensures the existence of an admissible operator $\mathcal{P}$. With existence established, the Geodesic Voronoi Tessellation (GVT) provides a feasible realization of $\mathcal{P}$ by designing the generator set $\mathcal{V}$. We define a generator set $\mathcal{V} = \{v_1, \dots, v_M\} \subset S$. For a star-shaped obstacle $O_k$, we assign a single generator $v_k$ in its vicinity to induce an annular cell. For a non-star-shaped obstacle $O_j$, a cluster $\mathcal{V}_j \subset \mathcal{V}$ ($|\mathcal{V}_j| \ge 2$) is distributed around $O_j$ to partition its surrounding area into multiple simply connected cells. By strictly adhering to the generator design, each subdomain $S_k$ is either Type I (Simply Connected) or Type II (Annulus containing exactly one obstacle $O_k$).  

\subsection{Local Partial Welding}  \label{sectionB}
For a non-star-shaped obstacle $O_k$ associated with multiple generators $\mathcal{V}_k$, GVT segments $S_k$ into multiple simply connected patches $\{S_{k,l}\}_{l=1}^L$, where $L = |\mathcal{V}_k| \ge 2$ denotes the number of generators assigned to $O_k$. It collectively encloses $O_k$. For notational simplicity, the subscript $k$ is omitted. Let $S_u, S_v \in \{S_{k,l}\}_{l=1}^L$ be two adjacent patches sharing a common interface $\gamma_{uv} = \partial S_u \cap \partial S_v$. And let $h_u: S_u \to \mathbb{D}$ and $h_v: S_v \to \mathbb{D}$ be their respective Riemann mappings to the unit disk $\mathbb{D}$~\cite{Riemannian geometry}. Since $h_u$ and $h_v$ are computed independently, the images of the common interface, $h_u(\gamma_{uv})$ and $h_v(\gamma_{uv})$ do not inherently coincide on the unit circle $\partial\mathbb{D}$. To seamlessly stitch 
the domain $\mathbb{D}_v$ to $\mathbb{D}_u$, we seek an optimal Möbius transformation $\Theta_{uv}^* \in \text{Aut}(\mathbb{D})$ that aligns the boundary values by minimizing the welding energy functional~\cite{welding}
\begin{equation} \label{lwe}
\Theta_{uv}^* = \arg\min_{\Theta\in\text{Aut}(\mathbb{D})} \int_{\gamma_{uv}}\left|h_u(p)-\Theta(h_v(p)) \right|^2 ds(p),
\end{equation}
where 
$$
\text{Aut}(\mathbb{D}) = \left\{ z \mapsto e^{i\phi} \frac{z - a}{1 - \bar{a}z} \mid \phi \in [0, 2\pi), a \in \mathbb{D} \right\}
$$ 
denotes the group of disk automorphisms, and $ds(p)$ represents the arc length element along the interface curve $\gamma_{uv}$. 
By propagating $\Theta_{uv}^*$ along the boundary values across all patches, we define the composite mapping $\Theta_{1 \to l}^* = \Theta_{1,2}^* \circ \dots \circ \Theta_{l-1,l}^*$ and a piecewise unified mapping $U: \bigcup_{l=1}^L S_l \to \vartheta_k$. For any point $p$ belonging to $S_l$, the map is explicitly given by
\begin{equation*} \label{unim}
U(p) = 
\begin{cases} 
h_1(p), &  l=1, \\
\Theta_{1,2}^* \circ h_2(p), &  l=2, \\
\Theta_{1 \to l}^* \circ h_l(p), &   l > 2.
\end{cases}
\end{equation*}
Since the union of patches $\{S_{k,l}\}_{l=1}^L$ surrounds the hole $O_k$, the resulting image \( {{\vartheta _k}} \) is a topological annulus. An example is provided in Appendix~\ref{app1}.

\subsection{Parallel Annulus Mapping}  \label{sectionC}
This step is applicable to both the Type II subdomains and the topological annulus \( {{\vartheta _k}} \) derived in Section~\ref{sectionB}. Let $\varpi _k$ denote the topological annulus enclosing the $k$-th obstacle $O_k$, which is defined as
\begin{equation*}
\varpi _k = 
\begin{cases} 
S_k, & \text{if } O_k \text{ is in a Type II subregion}, \\
\vartheta_k, & \text{if } O_k \text{ is surrounded by welded patches}.
\end{cases}
\end{equation*}
By leveraging the annulus conformal map algorithm in~\cite{multiconnected21}, we compute a conformal map $f_k: \varpi _k \to \mathbb{A}_k$ that transforms $\varpi _k$ onto a circular annulus
\begin{equation} \label{annulus}
\mathbb{A}_k = \{z \in \mathbb{C} \mid r_k < \|z\| < 1\},
\end{equation}
where $r_k \in (0,1)$ is the conformal modulus of $\varpi _k$, which is uniquely determined by the domain's geometry. 

\begin{remark}
Due to the domain decomposition established by $\mathcal{P}$, the global topological complexity is decoupled into $n$ independent annular sub-problems. Consequently, for MAS, the computation of $f_k$ and $r_k$ for each obstacle $O_k$ are performed in parallel, which allows agents to solve local mappings without global information. 
\end{remark}

\subsection{Global Partial Welding}  \label{sectionD}
After the parallel computation in Section~\ref{sectionC}, we obtain a set of $n$ disjoint annuli $\{\mathbb{A}_1, \dots, \mathbb{A}_n\}$. To assemble these local annuli while preserving the conformal structure, we introduce the partial welding operator $T_k$ for the $k$-th annulus. Since the relative position and orientation of the holes must be adjusted without distorting their local shapes, $T_k$ is restricted to be a similarity transformation $T_k(z) = \alpha_k z + \gamma_k$, $\alpha_k\in\mathbb{C} \setminus\{0\}$, $\gamma_k\in\mathbb{C}$, where $\alpha_k$ controls the rotation and scaling, and $\gamma_k$ determines the translation of the $k$-th annulus in the complex plane. Let $\iota_{ij} = \partial \varpi_i \cap \partial \varpi_j$ denote the shared interface between adjacent regions $\varpi_i$ and $\varpi_j$. The optimal configuration is determined by minimizing the global stitching error
\begin{equation} \label{eg}
{E_g} = \sum\limits_{(i,j) \in \mathcal{E}} {\int_{{\iota _{ij}}} {{{\left| {{T_i}({f_i}(p)) - {T_j}({f_j}(p))} \right|}^2}ds(p)} },
\end{equation}
where $\mathcal{E}$ represents the set of adjacency indices, and $ds(p)$ denotes the arc length element along $\iota_{ij}$. To eliminate global ambiguity, we set $T_1(z) = z$ as the reference frame ($\alpha_1=1$, $\gamma_1=0$). The minimization of~\eqref{eg} is a linear least-squares problem with respect to the parameters $\{\alpha_k, \gamma_k\}_{k=2}^n$, which can be solved efficiently~\cite{welding}. Solving this linear system yields the optimal parameters $\{\alpha_k^*, \gamma_k^*\}_{k=1}^n$ and the corresponding global welding $g_{k}: \varpi _k \to \mathbb{C}$:
\begin{equation}
g_{k}(p) = T^*_k(f_k(p)), \quad \forall  p \in \varpi_k.
\end{equation}
This optimization glues the discrete annuli $\{\mathbb{A}_1, \dots, \mathbb{A}_n\}$ into a coherent domain $\Xi = \bigcup_{k=1}^n T_k^*(\mathbb{A}_k)$.  

\subsection{Geometric Rectification}  \label{sectionE}
The global partial welding constructed in Section~\ref{sectionD} provides a set of piecewise mappings $\{g_k\}_{k=1}^n$. Although $g_k$ ensures topological invariance, discrete stitching errors and conformal distortions may persist. To generate a high-quality conformal mapping $\tau: S \to \Xi$, we execute a three-stage geometric rectification process. First, we strictly enforce the circularity of the inner boundaries by computing the canonical circle  (with center $c_k$ and radius $r_k$) for \({g_k}\left( {\partial {O_k}} \right)\).
The boundary mapping $\hbar_k $ is defined by projecting the boundary points onto these canonical circles as follows:
\begin{equation} \label{tbm}
\hbar_k(p) = c_k + r_k \frac{g_k(p) - c_k}{|g_k(p) - c_k|}, \quad \forall p \in \partial O_k.
\end{equation}
Second, to mitigate non-uniform area distortion, a global Möbius transformation $M(z, \varsigma ) = \frac{z - \varsigma}{1 - \bar{\varsigma}z}$ with $\varsigma \in \mathbb{D}$ is applied. We optimize the parameter $\varsigma$ to minimize the variance of the conformal factor $\lambda(z)$, thereby inflating the regions
\begin{equation}  \label{tfm} 
\varsigma^* = \arg \min_{\varsigma \in \mathbb{D}} \int_{\bigcup g_k(\varpi_k)} \left( \det(J_{M \circ \hbar}) - \bar{\lambda} \right)^2 dA(z), 
\end{equation}
where $\det(J_{M \circ \hbar})$ represents the area distortion factor, and $\bar{\lambda}$ is the mean conformal factor. We denote \(\ell  = M^*\circ \hbar \) with \({M^*} = M\left( {z,{\varsigma ^*}} \right)\). The preceding geometric rectifications inevitably introduce local angular distortions. To minimize angle distortions while maintaining the rectified circular boundaries, we solve the Laplace equation with Dirichlet boundary conditions determined by $\ell$. Using the Linear Beltrami Solver (LBS)~\cite{rectangular13}, the conformal mapping $\tau$ is computed by
\begin{equation} \label{thm}
\begin{cases}
\Delta \tau = 0, & \text{in } S, \\ \tau(p) = \ell(p), & \text{on}~\partial S.
\end{cases}
\end{equation}
This yields the diffeomorphism \(\tau\) that maps the poly-annulus \(S\) to the n-holed disk 
\(\Xi\). Then, we propose the Distributed Poly-Annulus Mapping Algorithm
(i.e., Algorithm~\ref{Algorithm3}), as illustrated in Fig.~\ref{Fig.2}. Finally, we prove that virtual center is computable in a distributed manner, and establish the topological conjugacy between the dynamical systems on $S$ and $\Xi$.  

\begin{proposition} \label{thm1}
The following claims hold for the conformal mapping $\tau: S \to \Xi$ constructed in Algorithm~\ref{Algorithm3}.
\begin{enumerate}
\item The virtual center $\mathfrak{c}^*$ is synchronized such that $\lim_{t \to D_{\mathcal{G}}} \hat{c}_j(t) = \mathfrak{c}^*, \forall j \in \mathcal{N}$. \label{thm1.1}
\item The stability of dynamics on $S$ is equivalent to that on $\Xi$, preserved by the scalar invariance $\dot{\mathcal{W}}(x) = \dot{\mathcal{V}}(\hat{\xi})$. \label{thm1.2}
\end{enumerate}
\end{proposition}

\begin{proof}
{For Claim \ref{thm1.1}),} let $\partial \Xi_{out}$ be the unit circle boundary in $\Xi$. For agents in $\mathcal{N}_A = \{ i \in \mathcal{N} \mid \partial \tau(S_i) \cap \partial \Xi_{out} \neq \emptyset \}$, they identify $\mathfrak{c}^*$ by measuring the constant curvature $\kappa_{out} = 1$ of their boundary arcs. For agents in $\mathcal{N}_I = \mathcal{N} \setminus \mathcal{N}_A$, the leader-following consensus protocol $\hat{c}_j(t+1) = (\text{deg}_j + 1)^{-1} (\hat{c}_j(t) + \sum_{l \in \mathcal{N}_j} \hat{c}_l(t))$ ensures $\mathfrak{c}^*$ propagates across the connected graph $\mathcal{G}$ within finite steps $D_{\mathcal{G}}$, where $\text{deg}_j$ denotes the node degree. 
For Claim \ref{thm1.2}), Since $\tau$ is a diffeomorphism on the compact manifold $S$, the pull-back dynamics satisfy $\dot{x} = J_{\tau}^{-1} \dot{\hat{\xi}}$. The time derivative of the induced Lyapunov function $\mathcal{W}(x) = \mathcal{V}(\tau(x))$ is derived as $\dot{\mathcal{W}}(x) = (\nabla_{\hat{\xi}} \mathcal{V})^T J_{\tau} \cdot \dot{x} = (\nabla_{\hat{\xi}} \mathcal{V})^T \dot{\hat{\xi}} = \dot{\mathcal{V}}(\hat{\xi})$. As $\mathcal{V}(\hat{\xi})$ ensures asymptotic stability in $\Xi$ ($\dot{\mathcal{V}} < 0$), it directly implies $\dot{\mathcal{W}}(x) < 0$ in $S$.
\end{proof}

Proposition~\ref{thm1} provides the theoretical justification for control design in the ball world \(\Xi\). The scalar invariance of the Lyapunov function ($\mathcal{W}(x) \equiv \mathcal{V}(\xi)$) ensures that the gradient flows in both regions are topologically conjugate. Consequently, complex collision avoidance and coverage tasks in poly-annulus $S$ are rigorously reduced to simplified problems suitable for sectorial partition in the $n$-holed disk $\Xi$.

\begin{algorithm}[t] 
\caption{\label{Algorithm3}Distributed Poly-Annulus Mapping Algorithm} 
\renewcommand{\algorithmicrequire}{\textbf{Initialize:}}
\renewcommand{\algorithmicensure}{\textbf{Finalize:}}
{\bf Initialize:} Poly-annulus $S$, obstacle set $\mathcal{O}$, generator set $\mathcal{V}$.
\begin{algorithmic}[1]
\State Partition $S$ into subdomains $\{S_k\}_{k=1}^M$ via GVT
\For{$k=1:M$}
\If{$O_k \subset S_k$ (Type II)}
\State Set topological annulus $\varpi_k \leftarrow S_k$
\Else
\State Construct $\varpi_k = \vartheta_k$ via local optimal welding $\Theta^*$
\EndIf \State {\bf end~if}
\State Compute local map $f_k: \varpi_k \to \mathbb{A}_k$ via \eqref{annulus}
\EndFor \State {\bf end~for}
\State Minimize $ E_{g}$ to obtain $\{\alpha_k, \beta_k\}_{k=2}^n$ via \eqref{eg}
\State Construct global welding $\{g_k = T_k \circ f_k\}_{k=1}^n$
\State Update boundary $\ell = M^* \circ \hbar$ via Möbius parameter $\varsigma^*$
\State Solve $\Delta \tau = 0$ in $S$ with $\tau|_{\partial S} = \ell$ via \eqref{thm}
\end{algorithmic}
{\bf Finalize:} Global Conformal Mapping $\tau: S \to \Xi$.
\end{algorithm}


\section{Theoretical results}\label{section4}
Based on the designed conformal mapping, this section establishes the diffeomorphic coverage control framework. We first propose a multi-hole partition algorithm that modifies the partition boundaries to ensure topological safety in the ball world. Subsequently, a length metric is introduced to design a distributed control law. Theoretical analysis on the existence of solutions and the stability of partition dynamics and agent kinematics is conducted. 
 
\subsection{Partition Dynamics}\label{PD}
This section proposes a distributed partition strategy 
in $n$-holed disk \(\Xi\) (the ball world) based on~\eqref{partitioning dynamic}. To adapt the sectorial partition to n-holed disk \(\Xi\), the partition bars must be modified to ensure collision avoidance while maintaining the topological connectivity of subregions. To analyze the impact of workload errors, we adopt the perturbation analysis against non-vanishing perturbations~\cite{khalil}. We distinguish between two geometric configurations: the linear partition bar $\hat{B}_i(\hat{\psi}_i)$ and the piecewise smooth curve $\hat{\Gamma}_i(\hat{\psi}_i)$ generated to bypass obstacles. This distinction allows us to model the actual dynamics of the workload $\hat{m}_i$ as the nominal dynamics (governed by $\hat{B}_i$) superimposed with a geometric disturbance term $\hat{\delta}_i$ arising from obstacle avoidance (i.e., $\dot{\hat{m}}_{i, act} = \dot{\hat{m}}_{i, nom} + \hat{\delta}_i$).

We propose the Multi-Hole Partition Algorithm (i.e., Algorithm~\ref{Algorithm1}) to transform $\hat{B}_i$ into $\hat{\Gamma}_i$ in a distributed manner. For every obstacle $\hat{O}_k$ intersected by the nominal bar $\hat{B}_i$, we define a buffer circle $\hat{\mathcal{C}}_{k,i}$ with radius $\hat{r}_{k,i} = \hat{r}_k(1 + i\beta)$. The parameter $\beta$ serves as a topological sequence factor, ensuring that the partition bars of different agents ($i \neq j$) are spatially separated when bypassing the same obstacle, thereby strictly preventing topological entanglements. The segment of partition bar inside the buffer circle $\hat{\mathcal{C}}_{k,i}$ is replaced by a circular arc $\hat{A}_{k,i}$. (The computational details in Appendix~\ref{app2}). The final curve $\hat{\Gamma}_i$ is constructed by concatenating these arcs $\hat{\mathcal{A}}_i = \{ \hat{A}_{k,i} \}_{k \in \hat{\mathcal{K}}_i}$ with the remaining linear segments $\hat{D}_i = [0, R] \setminus \bigcup_{k \in \hat{\mathcal{K}}_i} (\hat{p}_{in}^k, \hat{p}_{out}^k)$, where $\hat{\mathcal{K}}_i$ is the set of obstacles indices intersected by $\hat{B}_i$. These geometric modifications isolate partition bars from obstacles while strictly preserving the multi-agent topological order.

\begin{proposition}\label{pro1} 
Algorithm~\ref{Algorithm1} guarantees strict collision avoidance and topological invariance within $\Xi$. For any agents $i, j \in \{1, \dots, N\}$ ($i \neq j$) and any obstacle $\hat{O}_k$, the partition bars $\Gamma$ satisfy: $\hat{\Gamma}_i \cap \hat{O}_k = \emptyset$ and $\hat{\Gamma}_i \cap \hat{\Gamma}_j = \emptyset$.
\end{proposition}
\begin{proof}
Algorithm~\ref{Algorithm1} constructs the partition boundary $\hat{\Gamma}_i$ near obstacle $\hat{O}_k$ as a circular arc with radius $\hat{r}_{k,i} = \hat{r}_k(1 + i\beta)$. Since $i \ge 1$ and $\beta > 0$, the safety radius satisfies $\hat{r}_{k,i} > \hat{r}_k$. Thus, the bar is strictly separated from the obstacle boundary $\partial \hat{O}_k$ (where radius is $\hat{r}_k$), ensuring $\hat{\Gamma}_i \cap \hat{O}_k = \emptyset$. To prove $\hat{\Gamma}_i \cap \hat{\Gamma}_j = \emptyset$, we first demonstrate that the ordering of agent phases $\hat{\psi}_i$ is invariant under the partition dynamics \eqref{partitioning dynamic}. Let the set of sequentially ordered phases be $\mathcal{D}_{order} = \{ \hat{\boldsymbol{\psi}} \in [0, 2\pi)^N \mid \hat{\psi}_1 < \hat{\psi}_2 < \dots < \hat{\psi}_N \}$. Consider the angular gap between adjacent agents: $\Delta \hat{\psi}_i(t) = \hat{\psi}_{i+1}(t) - \hat{\psi}_i(t)$ (with $\hat{\psi}_{N+1} = \hat{\psi}_1 + 2\pi$). The evolution of the angular gap $\Delta \hat{\psi}_i$ follows the discrete Laplacian flow \(d\Delta {\hat\psi _i}/dt = {\dot {\hat\psi} _{i + 1}} - {\dot {\hat\psi} _i} = {k_{\hat \psi }}({\hat m_{i + 1}} - 2{\hat m_i} + {\hat m_{i - 1}})\). By the continuity of the density $\hat{\rho}$ on the compact manifold $S$, the workload $\hat{m}_i$ vanishes as the sectorial area $\hat{E}_i \to 0$ (i.e., $\Delta \hat{\psi}_i \to 0^+$). Given the total workload \(\sum\nolimits_{i = 1}^N {{m_i}} > 0\), a potential singularity where $\hat{m}_i = 0$ implies that its adjacent neighbors must retain non-vanishing workloads, (i.e., $\hat{m}_{i-1} + \hat{m}_{i+1} > 0$). Consequently, the boundary velocity of the gap satisfies
\begin{equation} 
\lim_{\Delta \hat{\psi}_i \to 0^+} \Delta \dot{\hat{\psi}}_i = k_{\hat{\psi}} (\hat{m}_{i+1} + \hat{m}_{i-1}) > 0. 
\end{equation} 
This strictly positive derivative ensures that $\mathcal{D}_{order}$ is a positively invariant set \cite{khalil}. The vanishing subregion induces a repulsive effect that prevents phase crossing, ensuring the bars never intersect along their linear segments. Then, consider two distinct agents $i$ and $j$ bypassing a same obstacle $\hat{O}_k$. Their spatial separation is determined by the difference in safety radii: $|\hat{r}_{k,i} - \hat{r}_{k,j}| = \hat{r}_k \beta |i - j|$. Since $i \neq j$ implies $|i-j| \ge 1$, the radial clearance is lower-bounded by $\hat{r}_k \beta > 0$. This ensures that even if two partition bars share the same angular interval during bypass, they remain disjoint in the radial dimension. Combining these two properties, we conclude $\hat{\Gamma}_i \cap \hat{\Gamma}_j = \emptyset$ for all $t \ge 0$, preserving the multi-agent topological order.
\end{proof}

With the strict topological safety and ordering invariance established, the modified partition bar $\hat{\Gamma}_i$ is proven to be geometrically valid within $\Xi$. However, the introduction of buffer arcs inevitably causes a deviation between $\hat{\Gamma}_i$ and $\hat{B}_i$. The following lemma quantifies the workload error $\hat{\delta}_i$ induced by this geometric regularization. 
\begin{lemma} \label{lemma 3.1}
The workload error $\Delta m_i$ for the $i$-th agent in Algorithm~\ref{Algorithm1} is strictly bounded and controllable via the parameter $\beta$. 
\end{lemma}
\begin{proof}
In Algorithm 2, when a partition bar $\hat{\Gamma}_i$ intersects with the obstacle $\hat{O}_k$, it is rerouted along the buffer circle with radius $\hat{r}_{k,i} = \hat{r}_k(1 + i\beta)$. The workload error contributed by $\hat{O}_k$ is given by 
$$
\Delta \hat{m}_{k,i} = \int_{\hat{\theta}_{s, k}}^{\hat{\theta}_{e, k}} \int_{\hat{r}_k}^{\hat{r}_{k,i}} \hat{\rho}(\hat{r}, \hat{\theta}) \hat{r} d\hat{r} d\hat{\theta},
$$ 
where $\Delta \hat{\theta}_k = \hat{\theta}_{e, k} - \hat{\theta}_{s, k}$ is the angular span. Since the density $\hat{\rho}$ is continuous on the compact set $S$, it is bounded by $\|\hat{\rho}\|_{\infty}$. Evaluating the inner radial integral yields $\int_{\hat{r}_k}^{\hat{r}_{k,i}} \hat{r} d\hat{r} = \hat{r}_k^2 i\beta + \frac{1}{2} \hat{r}_k^2 i^2 \beta^2$. For a small sequence factor $\beta \ll 1$, the second-order term $\beta^2$ is negligible. Summing over all obstacles $k \in \hat{\mathcal{K}}_i$ intersected by $\hat{\Gamma}_i$, and considering the worst-case agent index $i=N$, the total workload error satisfies
\begin{equation}
\Delta \hat{m}_i = \sum_{k \in \hat{\mathcal{K}}i} \Delta \hat{m}_{k,i} \le \beta \cdot N \cdot C_g \cdot \|\hat{\rho}\|_{\infty}, 
\end{equation}
where $C_g = \sum_{k=1}^n L_k \hat{r}_k$ is a positive constant and $L_k = \hat{r}_k \Delta \hat{\theta}_k$. By defining $C_{\hat{\delta}} = N C_g |\|\hat{\rho}\|_{\infty}$, we obtain $\Delta \hat{m}_i \le C_{\hat{\delta}} \beta$. Thus, the workload deviation $\Delta \hat{m}_i$ is controllable via $\beta$ and vanishes as $\beta \to 0$.  
\end{proof}

While the partition dynamics can ensure the topological safety of the virtual bars $\hat{\Gamma}_i$, a robust motion planning strategy is required for agents to navigate in their corresponding sub-regions $\hat{E}_i$. Note that Euclidean control laws may fail to prevent collisions with the intricate boundaries of the poly-annulus surface $S$. To address this, the following subsection constructs a safety-aware Riemannian metric 
to govern the agent kinematics.

\begin{algorithm}[t!]
    \caption{\label{Algorithm1} Multi-Hole Partition Algorithm}
    \renewcommand{\algorithmicrequire}{\textbf{Initialize:}}
    \renewcommand{\algorithmicensure}{\textbf{Finalize:}}
{\bf Initialize:} Agent number $N$, obstacles $\hat{\mathcal{O}}$, sequence factor $\beta$, radius of $n$-holed disk $R$. 
    \begin{algorithmic}[1]
    \For {$i=1:N$} 
    \State Update partition bar $\hat{B}_i(\hat{\psi}_i)$ with $\hat{\psi}_i$ via \eqref{partitioning dynamic}
    \State Identify $\hat{\mathcal{K}}_i = \{k\mid \hat{B}_i \cap \hat{O}_k \neq \emptyset\}$ 
    \For {each $k \in \hat{\mathcal{K}}_i$} 
    \State Set $\hat{\mathcal{C}}_{k,i}$ with $\hat{r}_{k,i} = \hat{r}_k (1 + i\beta)$
    \State Construct arc $\hat{A}_{k,i} \subset \hat{\mathcal{C}}_{k,i}$ s.t. $\partial \hat{A}_{k,i} = \hat{B}_i \cap \hat{\mathcal{C}}_{k,i}$.
    \EndFor \State {\bf end for}
    \State $\hat{\Gamma}_i = ([0, R] \setminus \bigcup_{k \in \hat{\mathcal{K}}_i} (\hat{p}_{in}^k, \hat{p}_{out}^k)) \bigoplus_{k \in \hat{\mathcal{K}}_i} \hat{A}_{k,i}$.
    \EndFor \State {\bf end for}
    \end{algorithmic}
    {\bf Finalize:} Modified partition bars $\hat{\Gamma} = \{ \hat{\Gamma}_1, \dots, \hat{\Gamma}_N \}$. 
\end{algorithm}

\subsection{Design of Distributed Controller}
This section establishes a geometric coordination mechanism for agent navigation on the poly-annulus $S$. First, a Riemannian length metric is constructed to encode obstacles avoidance. Second, a Remannian gradient-based control law is synthesized as a flow on the resulting manifold to ensure autonomous navigation.
\subsubsection{Design of Length Metric} \label{Design of Length metric}
In the mapping space $\Xi$, let $\hat{p} = (\hat{r}, \hat{\theta})$ be the coordinate in a local chart. The distance from $\hat{p}$ to the boundary of the $k$-th obstacle $\hat{O}_k$ (centered at $\hat{o}_k$) is $\hat{s}_k(\hat{p}) = \|\hat{p} - \hat{o}_k\| - \hat{r}_k$. To ensure the validity and smoothness of the navigation space, we construct a static weight function $\hat\sigma_0(\hat{p}) = \sum_{k=1}^n \left( e^{\mu \hat{s}_k(\hat{p})} - 1 \right)^{-1}$, where $\mu > 0$ is a curvature parameter. As $\hat{s}_k \to 0^+$, the weight $\hat\sigma_0 \to \infty$, which characterizes the infinite navigation cost at the boundaries. The Riemannian metric $\hat{\eta}_0$ on $\Xi$ is defined as
\begin{equation} \label{Metric_Xi}
\hat{\eta}_0(\hat{p}) = (1 + \hat\sigma_0(\hat{p}))^2 \text{diag}(1, \hat{r}^2).
\end{equation}
This metric is pulled back to the poly-annulus $S$ as $\eta_0(p) = \lambda^2(p) \hat{\eta}_0(\tau(p))$, where $\lambda(p) = \sqrt{|\det(J_\tau(p))|}$ is the conformal factor induced by the Jacobian $J_\tau = \partial \tau / \partial p$. The associated length metric $d_l(p, q)$ is defined as the infimum of path lengths:
\begin{equation} \label{dl} 
d_l(p, q) = \inf_{\gamma} \left\{ \int_0^1 \sqrt{ \dot{\gamma}(t)^T \eta_0(\gamma(t)) \dot{\gamma}(t) }  dt \right\}, 
\end{equation}
where $\gamma$ is a piecewise smooth curve such that $\gamma(0)=p$ and $\gamma(1)=q$. Notably, $d_l$ is defined solely based on static obstacles to ensure a time-invariant geometric workspace for the cost function~\eqref{cost}.

\subsubsection{Design of Control Law}
While the length metric $d_l$ governs global navigation, inter-agent collision avoidance is integrated via a dynamic metric. Let $\Omega_{ij} = \{p_i \in S \mid \|p_i - p_j\| \leq 2r_a\}$ denote the inter-agent collision set with agent radius $r_a$. Then define the feasible set $\mathcal{F}_i = S \setminus ( \bigcup_{k=1}^n O_k \cup \bigcup_{j \in \mathcal{N}_i} \Omega_{ij} )$. For the $i$-th agent, the composite weight is $\hat\sigma_i = \hat\sigma_0 + \sum_{j \in \mathcal{N}_i} \hat\sigma_{ij}$, where $\hat\sigma_{ij} = (e^{\mu (\|\hat{p}_i - \hat{p}_j\| - 2r_a)} - 1)^{-1}$ accounts for the proximity to neighbor $j$. Let $\hat{\eta}_i = (1+\hat\sigma_i)^2 \text{diag}(1, \hat{r}^2)$ be the local Riemannian metric in $\Xi$. To ensure the well-posedness of the induced dynamic manifold, we impose the following regularity condition.
\begin{assumption}[Safety Clearance] \label{Ass:Safety}
There exists a constant $\delta > 0$ such that the neighbor trajectories satisfy separation: $\hat{d}_{ij}\ge 2r_a + \delta$, $\forall t\ge 0$, $\forall j \in\mathcal{N}_i$, where ${{\hat d}_{ij}} = {d_{\hat l}}\left( {{{\hat p}_i},{{\hat p}_j}} \right)$. This condition restricts the system configuration to a compact subset of the feasible set $\mathcal{F}_i$, preventing metric singularities.
\end{assumption}
Based on this safety clearance assumption, the evolution of the Riemannian metric is mathematically bounded. Then, to ensure the effectiveness of the control input design, we need to analyze the relationship between the control input and the variation rate of the Riemannian metric.
\begin{lemma} \label{LemmaMR}
Under Assumption \ref{Ass:Safety}, the time derivative of the pull-back metric $\eta_i(t)$ satisfies $\|\dot{\eta}_i\|_F \le L_{\eta} \sum_{j \in \mathcal{N}_i} \|\dot{\hat{p}}_j\|$, where $L_{\eta} > 0$. Consequently, it asymptotically recovers a static manifold.
\end{lemma}
\begin{proof}
Applying the chain rule to $\eta_i = \lambda^2 (1+\hat\sigma_i)^2 \text{diag}(1, \hat{r}^2)$ leads to $\dot{\eta}_i = 2\lambda^2(1+\hat{\sigma}_i)\text{diag}(1, \hat{r}^2) \sum_{j \in \mathcal{N}i} (\partial \hat{\sigma}_{ij}/\partial \hat{d}_{ij}) \cdot (d\|\hat{p}_i - \hat{p}_j\|/dt)$. On the compact set $\mathcal{F}_i$, the terms $2\lambda^2(1+\hat\sigma_i)$ and $\text{diag}(1, \hat{r}^2)$ are bounded. By Assumption \ref{Ass:Safety}, the condition $\hat{d}{ij} \ge 2r_a + \delta$ ensures a uniform upper bound
\begin{equation*}
\begin{split}
\|\partial\hat\sigma_{ij}/\partial \hat{d}_{ij}\| &= \|-\mu e^{\mu(\hat{d}_{ij} - 2r_a)} (e^{\mu(\hat{d}_{ij} - 2r_a)} - 1)^{-2}\| \\
&\le \mu e^{\mu \delta} (e^{\mu \delta} - 1)^{-2} < \infty.
\end{split}    
\end{equation*}
Since $|d\|\hat{p}_i - \hat{p}_j\|/dt| \le \|\dot{\hat{p}}_i\| + \|\dot{\hat{p}}_j\|$,     
there exists a constant $L_{\eta}$ such that $\|\dot{\eta}_i\|_F \le L_{\eta} \sum_{j \in \mathcal{N}_i} \|\dot{\hat{p}}_j\|$, which implies $\lim_{t \to \infty} \dot{\hat{p}}(t) = 0 \implies \lim_{t \to \infty} \dot{\eta}_i(t) = 0$.
\end{proof}

In light of Lemma \ref{LemmaMR}, the control input $u_i$ is designed as the Riemannian gradient flow on the instantaneous manifold
\begin{equation} \label{u}
\mathbf{u}_i = - k_p \nabla_{\eta, p_i} J = - k_p \eta_i^{-1}(p_i) \frac{\partial J}{\partial p_i},
\end{equation}
where $k_p > 0$ is the gain and Riemannian metric $\eta_i$ is the pull-back of $\hat{\eta}_i = (1+\hat\sigma_i)^2 \text{diag}(1, \hat{r}^2)$. This formulation defines $u_i$ as a geometric invariant, where the transformation of $\eta_i$ naturally compensates for coordinate changes to maintain the consistency across local charts \cite{Unknown13}.

\begin{proposition} \label{Prop1}
The following properties hold for the control input~\eqref{u} on the feasible set $\mathcal{F}_i$.
\begin{enumerate}
\item {Non-degeneracy:} $\eta_i \in \mathbb{S}_{++}^2$, $\|\eta_i\| < \infty, \forall p_i \in \mathcal{F}_i$. \label{Prop1.1}
\item {Equivalence:} $\nabla_{\eta, p_i} J = 0 \iff \frac{\partial J}{\partial p_i} = 0$. \label{Prop1.2}
\item {Invariance:} $p_i(0) \in \mathcal{F}_i \implies p_i(t) \in \mathcal{F}_i, \forall t \geq 0$. \label{Prop1.3}
\end{enumerate}
\end{proposition}

\begin{proof}
{For Claim \ref{Prop1.1}),}
as $\tau$ is a conformal mapping, its Jacobian $J_\tau$ is non-singular, ensuring conformal factor $\lambda^2 = |\det(J_\tau)| > 0$ on poly-annulus $S$. Since $\hat\sigma_i$ is a composition of continuous exponential functions and is finite on the compact set $\mathcal{F}_i$. Thus, the scalar scaling $\lambda^2(1+\hat\sigma_i)^2$ is strictly positive and bounded, ensuring $\eta_i \in \mathbb{S}_{++}^2$ ($2 \times 2$ symmetric positive definite) and its induced matrix norm $\|\eta_i\|$ is finite $\forall p_i \in \mathcal{F}_i$.
{For Claim~\ref{Prop1.2}),}
$\eta_i$ is non-degenerate according to Claim~\ref{Prop1.1}). The Riemannian gradient $\nabla_{\eta, p_i} J = \eta_i^{-1} \frac{\partial J}{\partial p_i}$ is the image of the differential $dJ$ under the isomorphism $T^*S \to TS$. Since $\eta_i^{-1}$ is an invertible linear operator, its kernel is trivial, ensuring the equivalence of the critical points.
{For Claim \ref{Prop1.3}),}
as $p_i$ approaches the boundary $\partial \mathcal{F}_i$, the divergence $\hat\sigma_i \to \infty$ ensures that $\eta_i^{-1}$ vanishes. The asymptotic behavior is governed by $\lim_{p_i \to \partial \mathcal{F}i} \eta_i^{-1} = \lim_{\hat\sigma_i \to \infty} {(1+\hat\sigma_i)^{-2}} \text{diag}(1, \hat{r}^{-2}) /{\lambda^2}= \mathbf{0}$. The vanishing of the inverse metric $\eta_i^{-1}$ at $\partial \mathcal{F}_i$ nullifies the normal component of $u_i$. Following Nagumo's Theorem~\cite{khalil}, the trajectory is confined to $\mathcal{F}_i$ for all $t \geq 0$.
\end{proof}
Given the continuous-time dynamics \eqref{partitioning dynamic} and \eqref{u}, we need to address their distributed implementation. The following subsection introduces a numerical scheme that projects Riemannian gradient flows onto a discrete-time execution framework. 

\subsection{Poriferous Coverage Algorithm}
To approximate the optimal solution to the constrained optimization problem~\eqref{minJ}, we develop the Poriferous Coverage Optimization Algorithm (i.e., Algorithm~\ref{Algorithm2}). This algorithm bridges the continuous geometric theory with a distributed computational implementation. Initially, a quantization parameter $\varepsilon_p$ is defined to discretize the partition process into $K^* = \lceil 2\pi/\varepsilon_p \rceil$ intervals. For each $k \in \{1, \dots, K^*\}$, the index set $\chi_k = \{ i \in \mathbb{I}_N \mid i = \arg\min_j |\hat{\psi}_j - 2\pi(k-1)/K^*|\}$ identifies the anchor agent closest to the reference phase. During the execution horizon $t < T_\epsilon$, the dynamics are updated based on agent roles. For the anchor agent $a \in \chi_k$: The partition phase is fixed as $\hat{\psi}_a = 2\pi(k-1)/K^*$ to serve as a spatial reference. Its mapped position $\hat{p}_a$ and physical position $p_a$ evolve according to the corresponding Riemannian gradient \eqref{u}. For all other agents $i \notin \chi_k$: The partition phases $\hat{\psi}_i$ are updated via the dynamics \eqref{partitioning dynamic}. Simultaneously, their positions are synchronized across $S$ and $\Xi$ through the pull-back length metric~\eqref{dl} and the conformal mapping $\tau$. Furthermore, the global cost $J^k$ for each discrete initialization is evaluated via a distributed communication mechanism~\cite{zhai23}. Finally, the system selects $k^* = \arg \min_k J^k$ to finalize the optimal configuration $(\hat{\mathbf{\psi}}^*, \mathbf{p}^*)$. Algorithm \ref{Algorithm2} relies on the rigorous interplay between virtual partition bars and multi-agent motion. To certify the reliability, we provide a theoretical analysis on the existence of solutions, the input-to-state stability (ISS) of partition dynamics, and the asymptotic convergence of the closed-loop system.

\begin{algorithm}[t!]
\caption{\label{Algorithm2} Poriferous Coverage Optimization Algorithm}
\renewcommand{\algorithmicrequire}{\textbf{Initialize:}}
\renewcommand{\algorithmicensure}{\textbf{Finalize:}}
{\bf Initialize:} $\mathcal{N}$, $\tau$, $K^*$, $T_{\epsilon}$, and $J^* \leftarrow \infty$.
\begin{algorithmic}[1]
\For {$k = 1$ to $K^*$}
\State Set $\hat{\psi}_k \leftarrow 2\pi(k-1)/K^*$ 
\State Set $a \leftarrow \text{arg}\min_{j \in \mathcal{N}} |\hat{\psi}_j -\hat{\psi}_k|$
\While {$t < T_{\epsilon}$}
\For {each agent $i \in \mathcal{N}$}
\State $\hat{\psi}_i \leftarrow \hat{\psi}_k$ if $i=a$, else evolve $\dot{\hat{\psi}}_i$ via \eqref{partitioning dynamic}
\State Construct $\hat{\Gamma}_i$ in $\Xi$ and compute $u_i$ via \eqref{u}
\State Update $p_i$ via \eqref{dp} and synchronize $\hat{p}_i = \tau(p_i)$
\EndFor \State {\bf end~for}
\State $t \leftarrow t + \Delta t$
\EndWhile \State {\bf end~while}
\State Update $(J^*, \hat{\psi}^*, p^*) \leftarrow (J, \hat{\psi}, p)$ via \eqref{cost} if $J < J^*$
\EndFor \State {\bf end~for}
\end{algorithmic}
{\bf Finalize:} Optimal configuration $(\hat{\psi}^{*}, p^{*})$
\end{algorithm}

\subsection{Convergence Analyses}
This subsection focuses on the convergence of the proposed coverage control algorithm. At first, we analyze the existence of solutions to Problem~\eqref{minJ}. 

\begin{proposition}\label{the1}
There always exist optimal solutions to Problem~\eqref{minJ} in poly-annulus $S$, and these solutions are reachable for multi-agent dynamics~\eqref{dp} with control input~\eqref{u}.
\end{proposition}

\begin{proof}
The surface $S$ is a compact 2D Riemannian manifold with a smooth boundary, which implies the joint configuration space $\mathcal{Q} = S^N \times [0, 2\pi]^N$ is compact because it is the Cartesian product of compact sets. The cost functional $J$ in~\eqref{cost} is defined via the length metric $d_l$, which is induced by the continuous and positive definite pull-back metric $\eta_i$. Consequently, $J$ is a continuous functional on the compact set $\mathcal{Q}$. By the \textit{Weierstrass Extreme Value Theorem}~\cite{absil}, $J$ must attain its global minimum, which confirms the existence of the optimal configuration $(\hat{\mathbf{\psi}}^*, \mathbf{p}^*)$. For the global reachability, by the properties of cut loci on Riemannian manifolds~\cite{Cut loci}, the distance function $d_l$ is smooth almost everywhere, ensuring that the Riemannian gradient $\nabla_{\eta} J$ is well-defined. In addition, since $(S, \eta_i)$ is a compact manifold, it is geodesically complete according to the \textit{Hopf-Rinow Theorem}~\cite{Riemannian geometry}. Furthermore, the metric singularity $\hat{\sigma}_i \to \infty$ at the boundary $\partial \mathcal{F}_i$ induces a repelling potential that renders the interior of the feasible set forward invariant~\cite{khalil}. Thus, the system trajectories $\mathbf{p}(t)$ globally exist for all $t \in [0, \infty)$. Finally, consider the coverage cost $J$ as a Lyapunov candidate. Its time derivative along the system trajectories is given by:
\begin{equation}
\frac{dJ}{dt} = \sum_{i=1}^N \langle \nabla_{\eta, p_i} J, \dot{p}_i \rangle{\eta_i} = -k_p \sum_{i=1}^N | \nabla_{\eta, p_i} J |{\eta_i}^2 \leq 0.
\end{equation}
By \textit{LaSalle's Invariance Principle}~\cite{khalil}, since trajectories are bounded and $\dot{J} \leq 0$, the state asymptotically converges to the largest invariant set $\{ (\hat{\boldsymbol{\psi}}, \mathbf{p}) \mid \nabla{\eta} J = 0 \}$. Given that local minima are stable equilibria of the gradient flow, the system converges to the optimal coverage configuration.
\end{proof}
While Proposition~\ref{the1} ensures the existence of optimal configurations on $(S, \eta_i)$, the trajectories of physical system are governed by the coupled dynamics of $\dot{\hat{\psi}}_i$ and $\dot{p}_i$. A significant challenge arises from the geometric regularization in Algorithm~\ref{Algorithm1}, which introduces a non-vanishing perturbation $\hat\delta(t)$ to the partition dynamics. To solve this problem, the following theorem establishes the robustness of partition dynamics against this modification. 
\begin{theorem} \label{thee}
The workload error is exponentially input-to-state stable with respect to the sequence factor $\beta$, satisfying $\|\mathbf{e}(t)\| \le \|\mathbf{e}(0)\| e^{-\gamma t} + {C_{\hat{\delta}} \beta}/{\gamma}$, where $\mathbf{e}(t)=\mathbf{m}(t)-\bar{m}\cdot\mathbf{1}_N$ with \({\bf{m}}(t) = {\left( {{m_1}(t),\cdots,{m_N}(t)}\right)^T}\), and $\gamma$ is the exponential convergence rate. Consequently, the ultimate error bound can be made arbitrarily small by refining $\beta$. 
\end{theorem}
\begin{proof}
Construct the Lyapunov function as $V=\frac{1}{2}\|\mathbf{e}\|^2$. For the connected cycle graph $\mathcal{G}_N = (\mathcal{N}, \mathcal{E})$, the error dynamics obey $\dot{\mathbf{e}} = -k_{\hat{\psi}} \mathcal{L}_{\mathcal{G}}(\hat{\omega})\mathbf{e}$, where the weighted Laplacian $\mathcal{L}_{\mathcal{G}}$ is defined by
\begin{equation} \label{Laplacian}
[\mathcal{L}_{\mathcal{G}}]_{ij} =
\begin{cases}
\hat{\omega}_i + \hat{\omega}_{i+1}, & i = j \\
-\hat{\omega}_{\max(i,j)}, & (i, j) \in \mathcal{E} \\
0, & \text{otherwise}
\end{cases}
\end{equation}
The properties of $\mathcal{L}_{\mathcal{G}}$ are deatiled in Appendix~\ref{appLG}.
Under Assumption~\ref{Ass:Safety} ($\underline{\omega} > 0$), $\mathcal{L}_{\mathcal{G}}$ is symmetric positive semi-definite with $\text{ker}(\mathcal{L}_{\mathcal{G}}) = \text{span}(\mathbf{1}_N)$. Since $\mathbf{e} \perp \mathbf{1}_N$, invoking the \textit{Courant-Fischer Theorem}~\cite[Th. 4.2.2]{horn} yields 
$$
\dot{V}_{nom} = -k_{\hat{\psi}} \mathbf{e}^T \mathcal{L}_{\mathcal{G}}(\hat{\omega})\mathbf{e} \le -k_{\hat{\psi}} \underline{\omega} \lambda_2(\mathcal{L}_S) \|\mathbf{e}\|^2 \triangleq-\gamma\|\mathbf{e}\|^2,
$$ 
where $\mathcal{L}_S$ is the unweighted Laplacian of a cycle graph with algebraic connectivity $\lambda_2(\mathcal{L}_S) = 2(1 - \cos(2\pi/N))$ and $\gamma > 0$ denotes the exponential convergence rate.  Algorithm \ref{Algorithm1} introduces a disturbance $\|\hat{\delta}(t)\| \le C_{\hat{\delta}} \beta$. The actual derivative becomes $\dot{V} \le -\gamma \|\mathbf{e}\|^2 + C_{\hat{\delta}} \beta \|\mathbf{e}\|$. Splitting the dissipation with $\theta \in (0,1)$ gives 
$$\dot{V} \le -(1-\theta)\gamma \|\mathbf{e}\|^2 - \theta\gamma \|\mathbf{e}\| (\|\mathbf{e}\| - C_{\hat{\delta}} \beta / \theta\gamma).
$$ 
For $\|\mathbf{e}\| \ge C_{\hat{\delta}} \beta / \theta\gamma$, one has $\dot{V} \le -2(1-\theta)\gamma V$. By the \textit{Comparison Lemma}~\cite[Th. 4.19]{khalil}, the trajectory is bounded by 
$$
\|\mathbf{e}(t)\| \le \max \{ \|\mathbf{e}(0)\| e^{-(1-\theta)\gamma t}, C_{\hat{\delta}} \beta / \theta \gamma \}.
$$ 
By taking $\theta \to 1$, the error converges to 
$\limsup_{t\to\infty}\|\mathbf{e}(t)\|\le C_{\hat{\delta}}\beta/\gamma$, establishing exponential ISS with a controllable ultimate bound, where the adjustability of $\beta$ is rigorously guaranteed by Lemma~\ref{lemma 3.1}. 
\end{proof}



Based on the exponential ISS established in Theorem \ref{thee}, the workload error $e(t)$ is strictly bounded by the perturbation $\beta$. To ensure the global coverage performance, it is imperative that agent trajectories $p(t)$ in poly-annulus $S$ accurately track the evolving optimal configurations. Thus, we establish the stability of the Riemannian gradient flow \eqref{u} to guarantee the synchronization of physical states in the robot workspace with the virtual partitions in the ball world. 

\begin{theorem} \label{the3}
The control law \eqref{u} asymptotically stabilizes each agent to the Riemannian centroid $p_i^*$ of its subregion $E_i$ in poly-annulus $S$. Moreover, this equilibrium is locally exponentially stable.
\end{theorem}

\begin{proof}
Global asymptotic convergence of $p(t)$ to the critical set is established in Proposition \ref{the1}. Under the Riemannian gradient flow $\dot{p}_i = -k_p \nabla_{\eta, p_i} J$, the time derivative satisfies $\dot{J} = -k_p \sum_{i=1}^N \|\nabla_{\eta, p_i} J\|_{\eta_i}^2 \le 0$. The equilibrium condition $\nabla_{\eta, p_i} J = 0$ uniquely defines the set of Riemannian centroids $\{p_i^*\} \subset \mathcal{F}_i$. Crucially, the vanishing of $\eta_i^{-1}$ at $\partial \mathcal{F}_i$ ensures the forward invariance of the collision-free set $\mathcal{F}_i$, as proved in Proposition \ref{Prop1}.
To establish local exponential stability, we evaluate the positivity of the Hessian operator $\nabla^2 J$ at the equilibrium. According to Lemma \ref{LemmaMR}, the metric $\eta_i(t)$ asymptotically recovers a static manifold as partition dynamics~\eqref{partitioning dynamic} stabilizes. This justifies a quasi-static analysis of the agent dynamics. In $\Xi$. Under Assumption \ref{assf}, the Hessian $\nabla^2 \hat{J}_i$ for a fixed partition is given by 
$$
\nabla^2 \hat{J}_i(\hat{p}_i) = \int_{\hat{E}_i} (f''(d_{\hat{l}}) \nabla d_{\hat{l}} \otimes \nabla d_{\hat{l}} + f'(d_{\hat{l}}) \nabla^2 d_{\hat{l}}) \hat{\rho}(\hat{q}) d\hat{q}.
$$ 
Since $f' > 0$, $f'' > 0$, and $\nabla^2 d_{\hat{l}} \succ 0$ within the convexity radius of $\Xi$, it follows $\nabla^2 \hat{J}_i(\hat{p}_i^*) \succ 0$. By the pull-back $J = \hat{J} \circ \tau$, the Hessian at $p^*$ is transformed via the Jacobian of $\tau$: $\nabla^2 J(p^*) = (J_{\tau})^T \nabla^2 \hat{J}(\hat{p}^*) J_{\tau}$ with $J_{\tau}=\partial\tau/\partial p$. Since $\tau$ is a diffeomorphism, $J_{\tau}$ is non-singular. By \textit{Sylvester's Law of Inertia}~\cite{Riemannian geometry}, the positive definiteness of the operator is invariant under this congruence transformation, i.e., $\nabla^2 J_S(p^*) \succ 0$. Linearizing the Riemannian flow $\dot{p}_i = -k_p \nabla_{\eta, p_i} J$ around $p^*$ yields the system matrix $-k_p \nabla^2 J(p^*)$. Since $\nabla^2 J(p^*)$ is strictly positive definite, all eigenvalues lie in the open left-half complex plane, and $p_i^*$ is locally exponentially stable.
\end{proof}

The exponential ISS of partition dynamics~\eqref{partitioning dynamic} and the local stability of agent dynamics~\eqref{dp} ensure that the MAS tracks the Riemannian centroids $p_i^*$ with high fidelity. However, in practical implementations, the global coverage performance is influenced by the discrete properties of MAS and the finite sampling frequency. Based on the convergence results in Theorem~\ref{thee} and Theorem~\ref{the3}, we evaluate the gap between the actual cost $J(\hat{\psi}^A, p^A)$ and the theoretical optimum $J(\hat{\psi}^*, p^*)$ to establish the global quasi-optimality of the proposed algorithm.

\begin{theorem} \label{the4}
Algorithm \ref{Algorithm2} steers MAS to approximate the optimal configuration and ensures the coverage cost $J$ is minimized within 
an arbitrarily small tolerance.
\end{theorem}
\begin{proof}
Let $\mathcal{X} = (\hat{\psi}, p)$ denote the joint state on the product manifold $\mathbb{T}^N \times S^N$. By the triangle inequality, the cost deviation is decomposed into three distinct terms below
\begin{equation} \label{DJ}
\begin{aligned}
|J(\mathcal{X}^{*}) - J(\mathcal{X}^{A})| 
&\le  |J(\hat\psi^{*},p^{*}) - J(\hat\psi^{*},p^{k})| \\
& + |J(\hat\psi^{*},p^{k}) - J(\hat\psi^{k},p^{k})| \\
& + |J(\hat\psi^{k},p^{k}) - J(\hat\psi^{A},p^{A})| + \epsilon_{conv},\\
\end{aligned}
\end{equation}
where $\epsilon_{conv}$ represents the convergence residual due to the finite termination horizon $T_{\epsilon}$. {For the first term}, it quantifies the discrepancy between the continuous trajectory $p(t)$ and the discrete sequence $p^k$, which is updated via $p_{k+1} = \exp_{p_k}(-k_p \Delta t \nabla_{\eta} J)$~\cite{NB23}. Since the poly-annulus $S$ is compact, the forward invariance of $\mathcal{F}_i$ established in Proposition~\ref{Prop1} coupled with the \textit{ Hopf-Rinow Theorem}~\cite{Riemannian geometry} ensures that $S$ is geodesically complete, rendering the discrete update globally well-defined. Applying the second-order covariant Taylor expansion, the cost variation is dominated by the Riemannian Hessian $\nabla^2 J$. Invoking Assumption \ref{assf} ($f', f'' > 0$), the compactness of $S$ ensures a finite Hessian upper bound $L_{H} = \sup_{p \in S} \|\nabla^2 J(p)\|_l$, where $\|\cdot\|_l$ denotes the operator norm induced by the length metric $d_l(\cdot)$. $L_H$ explicitly incorporates the curvature of $S$ and the performance supremum $\bar{f}''=\sup{f}''$.Thus, one gets
\begin{equation}
|J(\hat{\psi}^{*}, p^{*}) - J(\hat{\psi}^{*}, p^{k})| \le 0.5 L_{H} \|\nabla_{\eta} J\|_{l, \max}^2 (\Delta t)^2,
\end{equation}
where the zeroth and first-order terms cancel exactly as the discrete retraction is intrinsically aligned with the Riemannian gradient flow.

{For the second term}, it stems from the quantization of the configuration torus $\mathbb{T}^N$ into $K^{*}$ intervals, which induces a geodesic deviation $d_{\mathbb{T}}(\hat{\psi}^{*}, \hat{\psi}^{k}) \le 2\pi/K^{*}$.  Invoking Assumption \ref{assf} ($f' > 0$), we apply the \textit{Stokes' Theorem}~\cite{Riemannian geometry} to the boundary $\partial \hat{E}_i$, the intrinsic gradient is $\nabla_{\hat{\psi}i} J = \left( f(d_l(p_i, q)) - f(d_l(p_{i-1}, q)) \right) \hat{\omega}(\hat{\psi}_i)$, where $\hat{\omega}(\hat{\psi}_i)$ denotes the marginal density at the $i$-th partition bar. Consequently, the quantization error is bounded by
\begin{equation}
|J(\hat{\psi}^{*}, p^{k}) - J(\hat{\psi}^{k}, p^{k})| \le L_{\hat{\psi}} (2\pi / K^{*}),
\end{equation}
where $L_{\hat{\psi}} = \sup_{\hat{\psi}} \|\nabla_{\hat{\psi}} J\|_l$ is constrained by $\bar{f}'=\sup{f}'$ and the maximum total variation of the Riemannian volume form across the boundaries $\partial \hat{E}_i$ in $\Xi$.

{For the third term}, it stems from the geometric regularization $\beta$ in Algorithm \ref{Algorithm2}. The rerouting of partition bars from $B_i$ to $\hat{\Gamma}_i$ induces a domain perturbation $\Delta \mathcal{M}_i \subset S$. By invoking \textit{Stokes' Theorem}~\cite{Riemannian geometry}, 
the cost deviation is
\begin{equation} \label{DJ_domain} 
\begin{aligned} |J(\mathcal{X}^{k}) - J(\mathcal{X}^{A})| &= \left| \sum_{i=1}^{N} \int_{\Delta \mathcal{M}i} f(d_l(p_i, q)) \rho(q) dV_l \right| \\ 
&\le L_g C_{\delta} \beta, 
\end{aligned} 
\end{equation}
where $L_g = \sup f(d_l)$ is the performance bound on $S$. In light of Lemma~\ref{lemma 3.1}, the workload deviation $|\Delta \hat{m}_i| \le C_{\delta} \beta$ linearly maps to the cost penalty via the conformal factor $\lambda^2$ of the pull-back volume form $dV_l = \det(\tau^* \hat{\eta}_0)^{1/2} dx$.
As established in Theorem \ref{the3}, $\nabla^2 J \succ 0$ at equilibrium ensures Geodesic Strong Convexity. For the performance function $f$ satisfying Assumption \ref{assf}, the quadratic growth condition on the manifold $S$ ~\cite{BS13} provides a rigorous bound on the state deviation in the length metric $d_l(\cdot)$:
\begin{equation}
d_l(\mathcal{X}^A, \mathcal{X}^*) \le \sqrt{2(J(\mathcal{X}^A) - J(\mathcal{X}^*)) / \lambda_{\min}(\nabla^2 J)},
\end{equation}
where $\lambda_{\min}(\nabla^2 J)$ is the minimum eigenvalue of the Riemannian Hessian.
\end{proof}

\subsection{Topological Universality}
The preceding analyses confirm the algorithm performance on the poly-annulus surface. However, the fundamental validity of this framework transcends specific obstacle shapes. As a result, we abstract the workspace into a Riemannian manifold $(\mathcal{M},\eta)$ and establish the maximal topological class for which the proposed diffeomorphic control is 
strictly applicable.

\begin{theorem} \label{the_topology}
The proposed framework is applicable to the maximal class of compact, orientable, 2-D Riemannian manifolds $(\mathcal{M}, \eta)$ with boundary, satisfying the Euler characteristic $\chi(\mathcal{M}) = 1 - n$ and Genus $g(\mathcal{M}) = 0$. For any $\mathcal{M}$ in this class, the distributed coverage problem is topologically conjugate to the harmonic gradient flow on the $n$-holed disk $\Xi$.
\end{theorem}

\begin{proof} The proof relies on De Rham Cohomology theory to establish the existence of the coordinate system required by partition dynamics~\eqref{partitioning dynamic} and agent model~\eqref{dp}. First of all, we prove the {cohomological existence of partition phase \(\psi\)}. The validity of sectorial partition depends on a global angular coordinate $\hat{\psi}$ that uniquely indexes the topology. In terms of differential forms, this requires a non-trivial closed 1-form $\omega \in H^1_{dR}(\mathcal{M})$. For Genus-0 manifolds with $n+1$ boundary components, $H^1_{dR}(\mathcal{M})$ is an $n$-dimensional vector space. By the \textit{Hodge Decomposition Theorem}~\cite{Gu08}, there exists a unique harmonic 1-form $\omega_h$ ($\Delta \omega_h = 0$). Then, the conformal mapping $\tau: \mathcal{M} \to \Xi$ is the holomorphic integration of $\omega_h$, where the partition phase $\hat{\psi}$ is the period integral of the pull-back form $\tau^* \hat{\omega}$. This ensures that boundaries $\partial E_i$ are harmonic level sets, maintaining orthogonality to the Riemannian gradient flow. Then, we analyze the {topological conjugacy of agent dynamics}. Agent motion is governed by the vector field $X = -\nabla_{\eta} J$ on $\mathcal{M}$. The conformal diffeomorphism $\tau$ induces an isomorphism between tangent bundles $T\mathcal{M}$ and $T\Xi$. The push forward operator $\tau_*$ maps the gradient field to $X_{\Xi} = -\nabla_{\hat{\eta}} \hat{J}$ on $\Xi$ via $X_{\Xi} = \lambda^{-2} \tau_* X$, with $\lambda^2$ being the conformal factor. By the \textit{Poincaré-Hopf Theorem}~\cite{Gu08}, the sum of vector field indices equals $\chi(\mathcal{M})$. Since $\chi(\mathcal{M}) = \chi(\Xi) = 1 - n$, the equilibrium set (centroids) remains a topological invariant. Thus, the convergence established in Theorem \ref{the3} on $\Xi$ extends to any $\mathcal{M}$ within this class, regardless of boundary irregularity. Finally, we analyze the {genus constraint}. The restriction to genus $g(\mathcal{M}) = 0$ is rigorous. If $g(\mathcal{M}) > 0$, the non-commutative fundamental group $\pi_1(\mathcal{M})$ introduces handle cycles in $H^1_{dR}(\mathcal{M})$ that do not correspond to obstacles. A 1-dimensional angular ordering $\hat{\psi} \in [0, 2\pi)$ fails to characterize the state relative to these handles without complex branch cuts. Thus, $\chi(\mathcal{M}) = 1 - n$ with Genus 0 defines the upper bound for 1-dimensional sectorial coverage.
\end{proof}

The theoretical analysis constitutes a rigorous chain of logical deductions, guaranteeing the system reliability across geometric, dynamic, and computational implementation. First, Proposition \ref{the1} establishes the topological conjugacy and existence of the coordinate system via the conformal mapping $\tau$, while the metric pull-back $\eta = \tau^* \hat{\eta}_0$ preserves the intrinsic geometry of the poly-annlus. Building on this, Lemma \ref{lemma 3.1} and Proposition \ref{pro1} confirm the strict geometric safety of sectorial partitions, ensuring that the boundaries $\hat{\Gamma}_i$ are harmonic level sets that remain valid under constraints. Regarding control feasibility, Proposition \ref{Prop1} guarantees the non-degeneracy and forward invariance of the feasible set $\mathcal{F}_i$, ensuring that the Riemannian gradient flow is globally well-defined and collision-free. This is further supported by Lemma \ref{LemmaMR}, which proves the metric regularity of the pull-back manifold. This indicates 
that it asymptotically recovers a static structure as the agents converge. Subsequently, the stability analysis reveals a hierarchical convergence structure. To be specific, Theorem \ref{thee} proves that the virtual partition dynamics $\hat{\psi}$ exhibit exponential ISS against geometric perturbations $\beta$, as supported by the contraction properties in Lemma \ref{lemma 3.1}. Coordinated with the robust partitions, Theorem \ref{the3} guarantees that the multi-agent kinematics $p$ asymptotically converge to the optimal Riemannian centroids with local exponential rates. Theorem \ref{the4} bridges the gap between continuous theory and discrete implementation by validating that the numerical outputs strictly approximate the theoretical optimum $\mathcal{X}^*$ within a bounded error determined by $(\Delta t, K^*, \beta)$. Finally, Theorem \ref{the_topology} establishes the topological universality of the framework and demonstrates that the gradient flow remains valid for the maximal class of Genus-0 manifolds, regardless of metric distortion and boundary complexity.  


\section{Numerical Simulation} \label{section5}

To validate the effectiveness of the coverage control appraoch (Algorithms \ref{Algorithm3}, ~\ref{Algorithm1}, and ~\ref{Algorithm2}), numerical simulations are conducted on the poriferous region $S$ embedded in $\mathbb{R}^3$, which is modeled as a triangular mesh with 1844 vertices.

\subsection{Nominal Coverage Performance}
In the nominal scenario, the MAS consists of $N=6$ agents tasked with covering the poriferous region $S$ with a non-uniform density $\rho(\theta, r) = \exp(\sin^2\theta + \cos\theta) + 0.01r$. Based on~\eqref{Metric_Xi}, we obtain the Riemannian metric \(\hat\eta\) of the ball world \(\Xi\). Through conformal mapping \(\tau\), we obtain the Riemannian metric of poriferous region \(S\). Furthermore,  the length metric in \(\Xi\) and \(S\) are obtained by~\eqref{dl}. The simulation is initialized with the parameters listed in Table~\ref{table_parameters} and the selection criteria for simulation parameters are provided in Appendix~\ref{apc}. The evolution of the MAS in $\Xi$ and $S$ are illustrated in Fig.~\ref{Fig.3} and Fig.~\ref{Fig.4}, respectively. In these snapshots, the black radial lines represent the boundaries of the sectorial partition, green stars indicate the optimal centroids $\mathbf{c}_i^*$, and circles denote the current agent positions $p_i$. The sectorial partition boundaries successfully bypass the topological holes, ensuring each subregion remains simply connected. As shown in Fig.~\ref{Fig.5}, the workload converges to an equilibrium of $1.0477$ and the phase angles converge to constants. The control inputs asymptotically vanish, which confirms that the control law~\eqref{u} drives the agents to their optimal positions. According to Theorem~\ref{the4}, these positions are sufficiently close to the optimal centroids. It is worth noting that the analysis of algorithm computational complexity and the scalability of MAS is provided in Appendix~\ref{app3}. The simulation results illustrate the effectiveness of the proposed coverage control framework.
\begin{table}[!t]
\renewcommand{\arraystretch}{1}
\caption{Simulation Parameters and Configurations}
\label{table_parameters}
\centering
\begin{tabular}{ccl}
\toprule
\textbf{Symbol} & \textbf{Value} & \textbf{Physical Meaning} \\
\midrule      
$N$             & 6              & Total number of agents \\
$K^*$           & 30             & Centroid tracking gain \\
$T_\varepsilon$ & 40             & Convergence threshold \\
$k_{\hat{\psi}}$& 0.2            & Partition gain \\
$k_p$           & 0.12           & Agent gain \\
$\beta$         &0.005           & Sequence factor \\
\bottomrule
\end{tabular}
\end{table}
\begin{figure}[t!]
\centering
\includegraphics[width=0.48\textwidth]{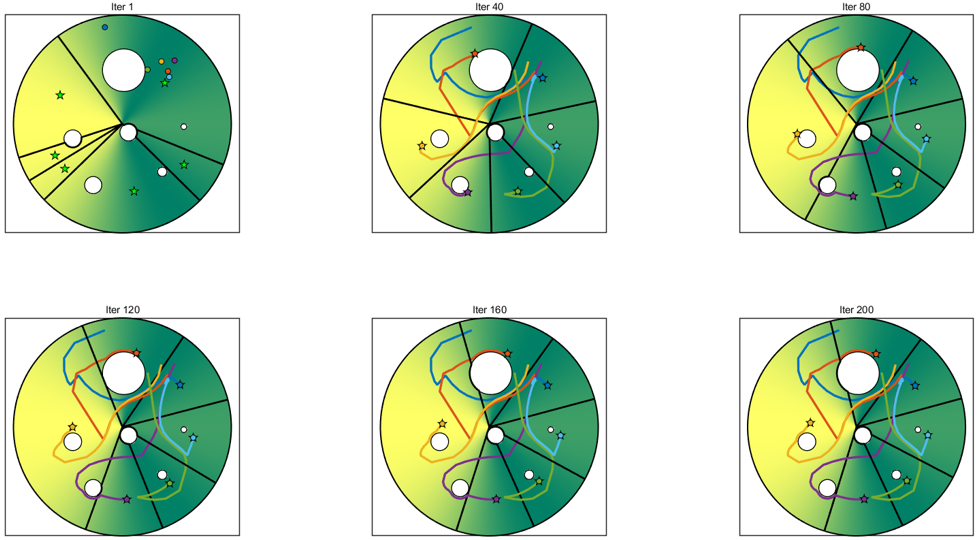}
\caption{\label{Fig.3} Temporal evolution of the MAS in the ball world $\Xi$. Black radial lines denote the safe sectorial partitions, circles represent agent positions, and green stars indicate the optimal Riemannian centroids. The partition bars strictly bypass the circular obstacles, ensuring topological safety. } \label{Fig.3}
\end{figure}

\begin{figure}[t!]
\centering
\includegraphics[width=0.48\textwidth]{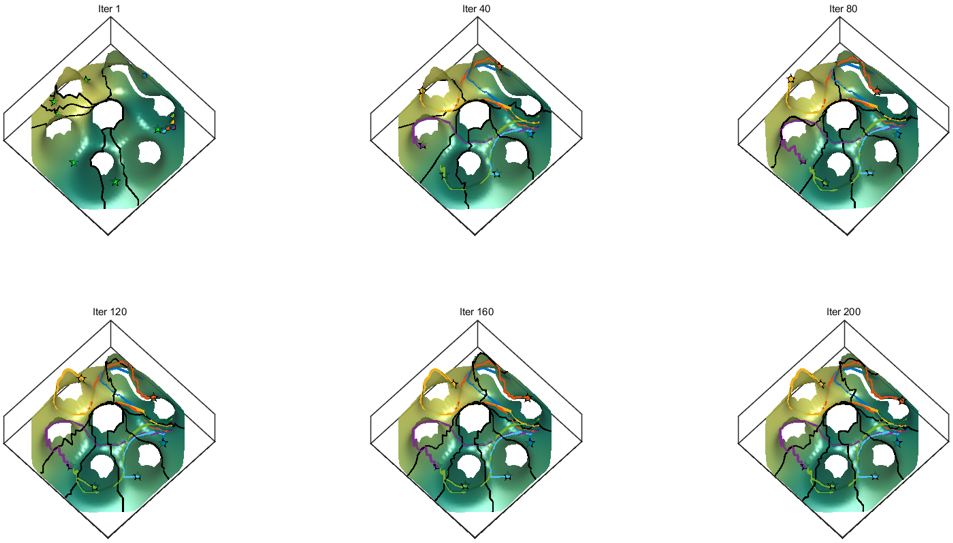}
\caption{\label{Fig.4} Corresponding coverage evolution in the poly-annulus $S$. Through the inverse conformal mapping $\tau^{-1}$, the linear sectors in $\Xi$ are transformed into curvilinear regions in $S$ that naturally conform to the irregular obstacle boundaries while maintaining simply connected topologies.} \label{Fig.4}
\end{figure}

\begin{figure}[t!]
\centering
\includegraphics[width=0.48\textwidth]{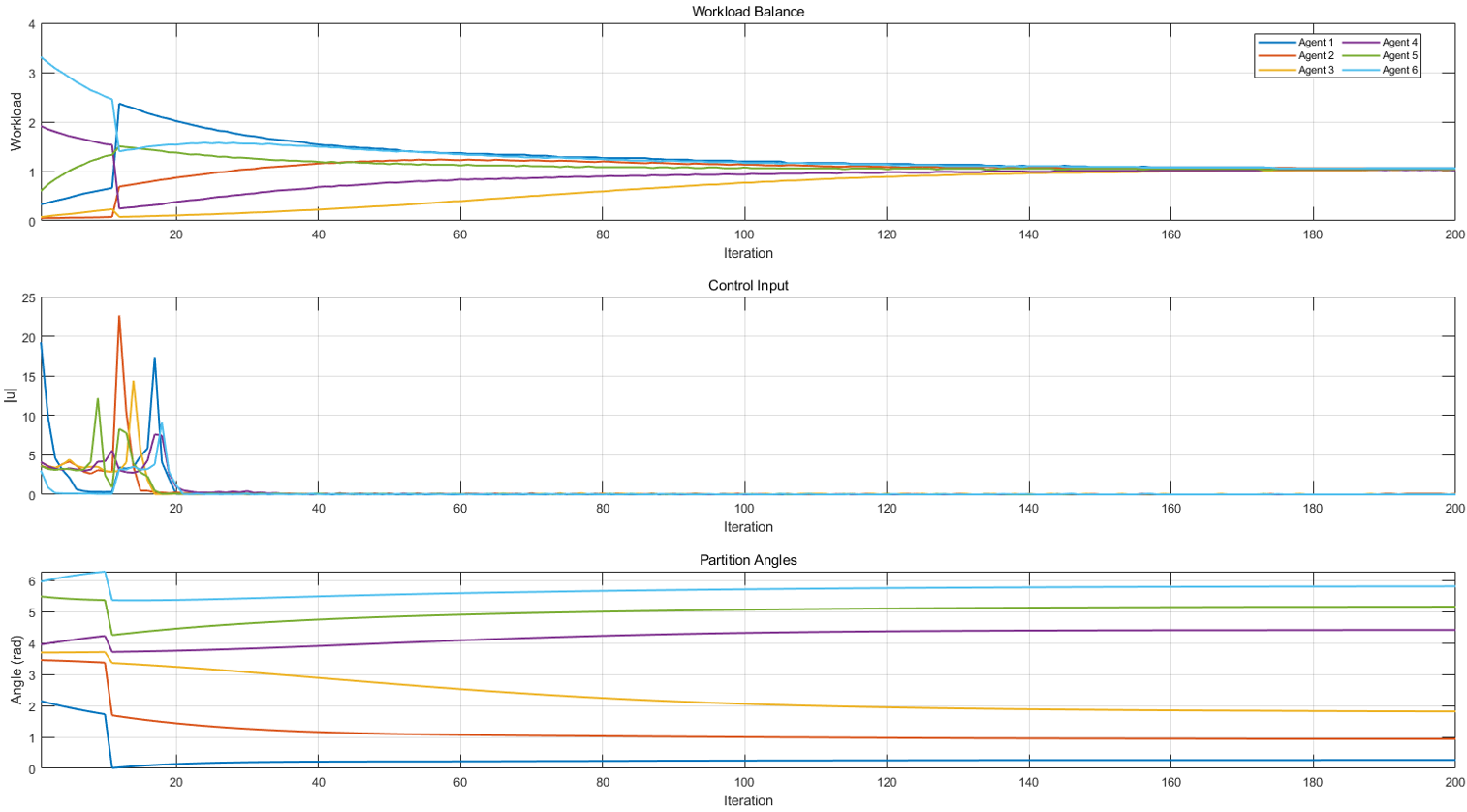}
\caption{\label{Fig.5} Evolution of performance indices under nominal conditions: (Top) Convergence of individual agent workloads $m_i$ to the equilibrium; (Middle) Asymptotic decay of control inputs $u_i$, indicating arrival at optimal centroids; (Bottom) Stabilization of partition phases $\psi_i$.} \label{Fig.5}
\end{figure}
\subsection{Robustness under Agent Failure}
To evaluate its robustness, a hardware contingency is simulated at $Interation=200$, where Agents 4, 5, and 6 cease operation. Upon failure, the remaining active agents $\mathcal{A}$ must redistribute the workload previously managed by the failed agents. As shown in the latter half of Fig.~\ref{Fig.6} and Fig.~\ref{Fig.7}, immediately following the failure, the partition boundaries automatically expand to close the gap. And the coverage performance indexes under agent failures is illustrated in Fig.~\ref{Fig.8}. In the normal situation, the workload converges to an equilibrium value of approximately $1.0477$ , and the phase angles converge to constants.  The control input $u$ for all agents asymptotically decays to zero. Following the failure of half of the agents, the workload $m$ converges to an equilibrium value of approximately $2.0953$ and the phase angles $\psi$ also converge to constants. The control input $u$ for all agents asymptotically decays to zero again. Although the workload per agent increases to a new equilibrium, the system maintains stability. Notably, the stability analysis of the system under agent failures is provided in Appendix ~\ref{app3}. This demonstrates that the distributed framework can autonomously re-balance the workload without requiring global reconfiguration of the conformal map $\tau$.

The numerical performance of the proposed coverage control approach is summarized in Table~\ref{table_metrics}, which evaluates the system across three distinct stages: the Nominal case (nominal operation with $N=6$), and the Pre-Fault and Post-Fault stages (contingency scenario with $N=6$ and $N=3$, respectively). And the indices presented include: $\bar{m}$ (i.e., the average workload assigned to each active agent), Imbalance (i.e., the workload distribution ratio $\max(m)/\min(m)$), RMSE (i.e., the root mean square error of the workload distribution relative to the equilibrium), $n$ (i.e., the number of convergence iterations) and $t_{avg}$ (i.e., the average computation time per agent per iteration).

\begin{table}[!h]
\renewcommand{\arraystretch}{1.3}
\caption{Summary of Numerical Performance Indices}
\label{table_metrics}
\centering
\footnotesize 
\begin{tabular*}{\columnwidth}{@{\extracolsep{\fill}}lcccccc}
\toprule
\textbf{Scenario} & \bm{$\bar{m}$} & \textbf{Imbalance} & \textbf{RMSE} & \textbf{Iters ($n$)} & \bm$t_{avg}$ \textbf{(s)} \\
\midrule
Nominal    & 1.0477 & 1.0017 & $6.75 \times 10^{-3}$ & 134 & 0.0196 \\
Pre-Fault  & 1.0477 & 1.0016 & $5.91 \times 10^{-3}$ & 115 & 0.0199 \\
Post-Fault & 2.0953 & 1.0009 & $8.21 \times 10^{-3}$ & 76  & 0.0213 \\
\bottomrule
\end{tabular*}
\end{table}

\begin{figure}[t!]
\centering
\includegraphics[width=0.48\textwidth]{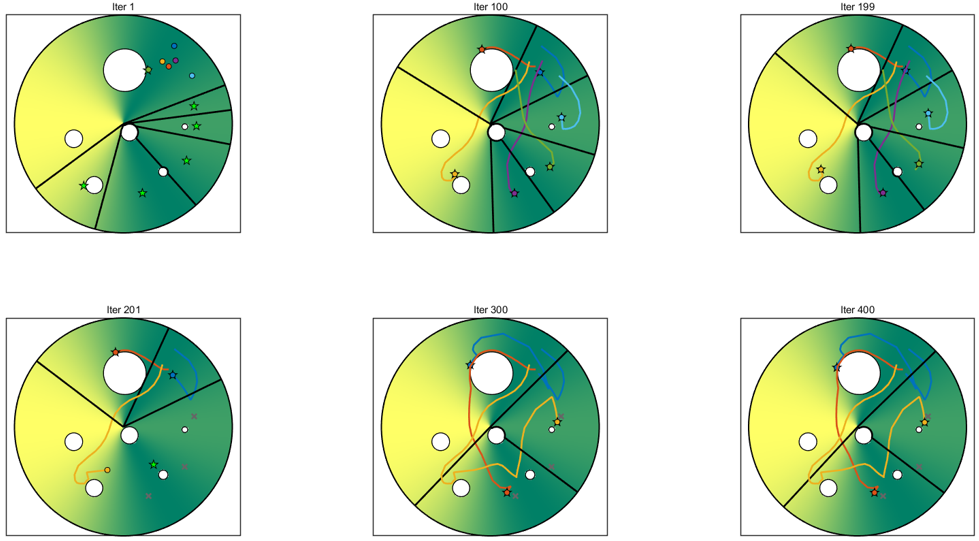}
\caption{\label{Fig.6} Self-healing response in $\Xi$ under agent failure at $t=200$. The active agents autonomously expand their sectorial angles to cover the regions vacated by the failed agents, maintaining global coverage continuity.} \label{Fig.6}
\end{figure}

\begin{figure}[t!]
\centering
\includegraphics[width=0.48\textwidth]{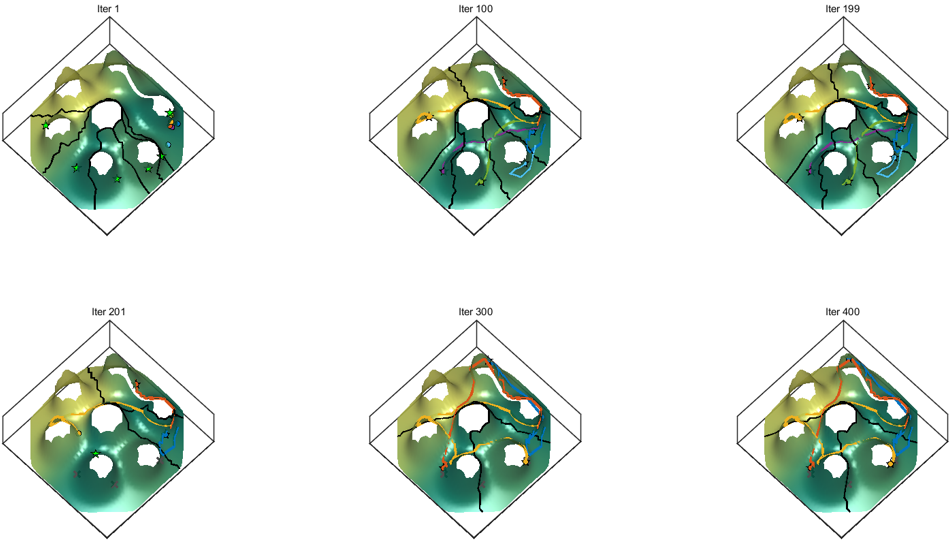}
\caption{\label{Fig.7} Reconfiguration of the MAS in $S$ following agent failure. The remaining agents navigate across the manifold to redistribute the global workload, demonstrating robustness without requiring global re-mapping.} \label{Fig.7}
\end{figure}

\begin{figure}[t!]
\centering
\includegraphics[width=0.48\textwidth]{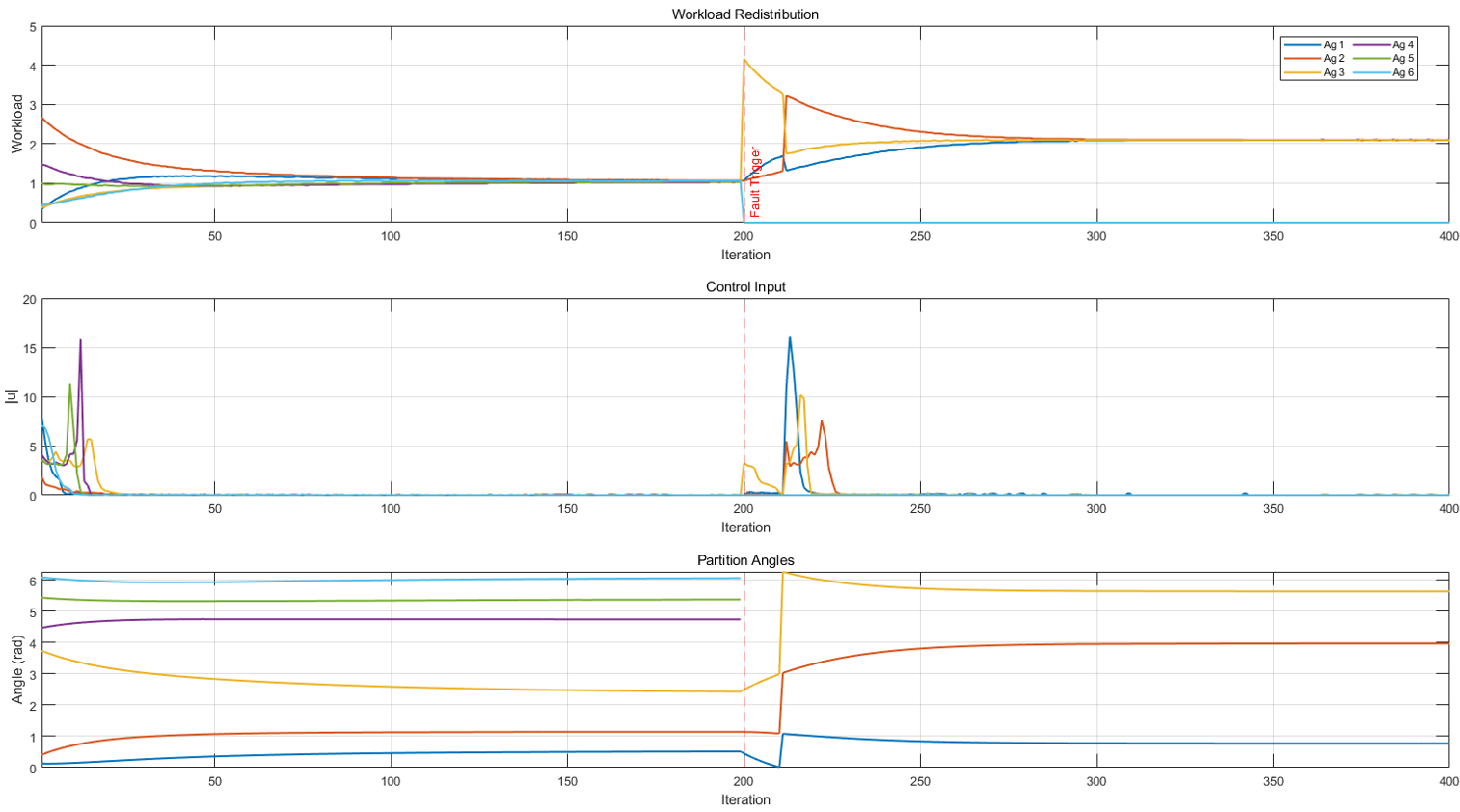}
\caption{\label{Fig.8} Performance metrics during the failure recovery process. The dashed line at $t=200$ marks the fault occurrence. The system exhibits a rapid transient response, where workloads $m_i$ re-converge to a new, higher equilibrium, verifying the automatic rate compensation capability.} \label{Fig.8}
\end{figure}

\subsection{Comparison with Baseline Method}
Finally, we compare the proposed approach with the Hierarchical Coverage Control (HCC) algorithm~\cite{porous}, which relies on Voronoi-based partition. As illustrated in Fig.~\ref{Fig.9}, when applied to the poly-annulus manifold $S$, the traditional Euclidean Voronoi partition fails to account for the poriferous topology. This leads to topological disconnection, where multiple obstacles or holes are trapped within a single Voronoi cell, or a cell is split into disconnected components by a hole. Consequently, the optimal centroids fall inside obstacles or in inaccessible regions across a hole. Under these conditions, MAS models with non-holonomic constraints cannot achieve collision avoidance or reach the assigned centroids, leading to local minima. In contrast, our sectorial partition with conformal mapping ensures that each subregion remains connected and topologically simple, effectively bypassing the inherent defects of Voronoi partition in non-convex poriferous regions.
\begin{figure}[t!]
\centering
\includegraphics[width=0.48\textwidth]{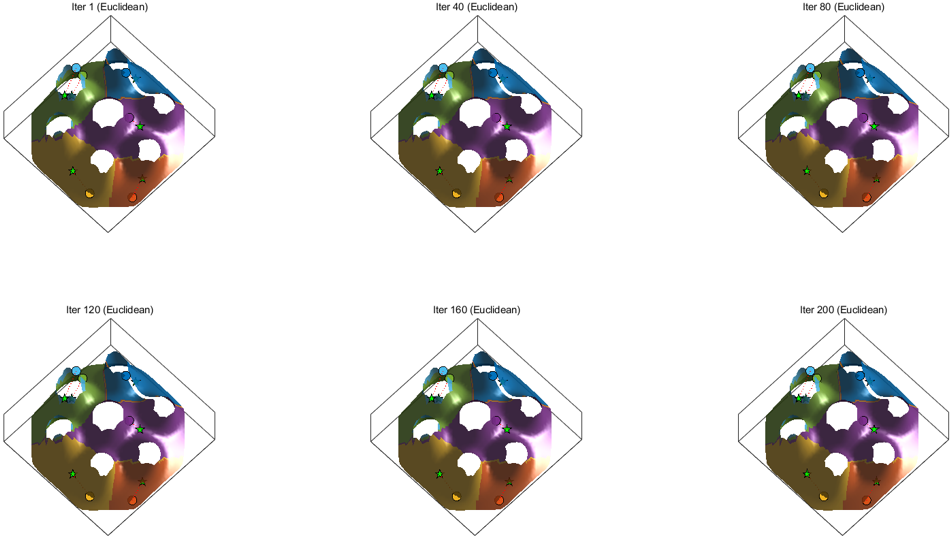}
\caption{\label{Fig.9} Failure case of the baseline Voronoi partition~\cite{porous} on $S$. The metric distortion causes topological disconnection (e.g., a single cell spanned across holes), leading to unreachable centroids trapped inside obstacles.} \label{Fig.9}
\end{figure}

\section{Conclusion} \label{section6}
This paper developed a diffeomorphic control framework that rigorously resolved the topological obstacles inherent in multi-agent coverage on poriferous surfaces. By constructing a distributed poly-annulus conformal mapping, we established a topological conjugacy that transform arbitrary multiply-connected region into the $n$-holed disk. Based on the conformal mapping, the closed-loop dynamics exhibited a hierarchical convergence. The partition dynamics achieved exponential Input-to-State Stability against geometric perturbations, while the agent dynamics asymptotically converged to the optimal Riemannian centroids. Crucially, these global guarantees were achieved through a fully distributed domain decomposition to parallel partial welding, ensuring the framework's universality and scalability across a broad class of Genus-0 surfaces. Ultimately, this work proved that intricate non-convex coverage tasks could be resolved via intrinsic gradient flows on regularized coordinate systems. Future work will focus on the time-varying environments.


\appendices
\section{Partial Welding Example}\label{app1}  
To illustrate the mathematical mechanism of partial welding, we present a intuitive example. Reconstruct the topological connectivity of a slit domain $\mathcal{D} = \mathbb{C} \setminus [-i, i]$ by welding its separated boundaries.
Let the slit $[-i, i]$ be regarded as two distinct boundary segments: $\gamma^+ = \{ iy + \epsilon \mid y \in [-1, 1] \}$, and $\gamma^- = \{ iy - \epsilon \mid y \in [-1, 1] \}$, where $\epsilon \to 0^+$. The welding mapping $T: \mathcal{D} \to \Omega$ is defined by the elementary function $w = T(z) = \sqrt{z^2 + 1}$. This mapping transforms the slit domain $\mathcal{D}$ into the $\Omega = \mathbb{C} \setminus (-\infty, 0]$. Crucially, it enforces the boundary identification condition:
\begin{equation}
T(iy^+) = T(iy^-) = \sqrt{1-y^2}, \quad \forall y \in [-1, 1].
\end{equation}
This implies that the disjoint banks $\gamma^+$ and $\gamma^-$ are mapped to the exact same locus in $\Omega$, thereby welding the cut back together. In our proposed framework, the optimization of $\Theta_{uv}$~\eqref{lwe})serves as the numerical inverse operator of such welding. It explicitly identifies the boundary correspondence:
\begin{equation}
h_u(\partial S_u) \cong \Theta_{uv}(h_v(\partial S_v)),
\end{equation}
which effectively restoring the continuous topology of poly-annulus $S$ from the decomposed local disks.

\section{Geometric Computation of Safe Boundary}\label{app2} 
This appendix details the numerical generation of the partition bar $\hat{\Gamma}_i$ in $\Xi$, consisting of linear segments and circular arcs. For intersection identification, the $i$-th nominal partition bar is parameterized as a vector $\hat{\mathbf{x}}(\hat\xi) = \hat\xi \hat{\mathbf{b}}_i$, where $\hat\xi \in [0, 1]$. The $k$-th obstacle is defined by its center $\hat{o}_k$ and radius $\hat{r}_k$. The intersection points are determined by the Euclidean constraint $\|\hat{\mathbf{x}}(\hat\xi) - \hat{o}_k\|^2 = \hat{r}_k^2$, which yields the quadratic equation:
\begin{equation}
|\hat{\mathbf{b}}_i|^2 \hat\xi^2 - 2 (\hat{\mathbf{b}}_i \cdot \hat{o}_k) \hat\xi + (|\hat{o}_k|^2 - \hat{r}_k^2) = 0.
\end{equation}
Solving for $\hat\xi$ gives the entry and exit parameters $\hat\xi_{in} < \hat\xi_{out}$, with corresponding coordinates $\hat{p}_{in} = \hat\xi_{in} \hat{\mathbf{b}}_i$ and $\hat{p}_{out} = \hat\xi_{out} \hat{\mathbf{b}}_i$. To bypass the obstacle, a circular arc is generated on the buffer circle with radius $\hat{r}_{k,i} = \hat{r}_k(1 + i\beta)$. The entry and exit angles relative to $\hat{o}_k$ are computed as $\hat\psi_{in} = \text{atan2}(\hat{p}_{in} - \hat{o}_k)$ and $\hat\psi_{out} = \text{atan2}(\hat{p}_{out} - \hat{o}_k)$. The safe arc $\hat{A}_{k,i}$ is then discretized via angular interpolation:
\begin{equation}
\hat{\mathbf{x}}_{arc}(u) = \hat{o}_k + \hat{r}_{k,i} \begin{bmatrix} \cos(\hat\psi_{in} + u \frac{\Delta \tilde{\hat\psi}}{U-1}) \ \\sin(\hat\psi_{in} + u \frac{\Delta \tilde{\hat\psi}}{U-1}) \end{bmatrix},
\end{equation}
where $u \in \{0, \dots, U-1\}$. This ensures $\hat{\Gamma}_i$ is piecewise smooth and strictly disjoint from obstacles.

\section{Spectral Properties of $\mathcal{L}_{\mathcal{G}}(\hat{\omega})$} \label{appLG}
For the weighted Laplacian $\mathcal{L}_{\mathcal{G}}(\hat{\omega})$ defined in \eqref{Laplacian}, under $\hat{\omega}_k \ge \underline{\omega} > 0, \forall k \in \mathcal{N}$, the properties hold:
\begin{enumerate}
\item $\mathcal{L}_{\mathcal{G}} = \mathcal{L}_{\mathcal{G}}^T$ and $\mathcal{L}_{\mathcal{G}} \succeq 0$.
\item $\text{ker}(\mathcal{L}_{\mathcal{G}}) = \text{span}(\mathbf{1}_N)$, implying $\sum_{i=1}^N \dot{\hat{m}}_i = 0$.
\item $\lambda_2(\mathcal{L}_{\mathcal{G}}) \ge \underline{\omega} \lambda_2(\mathcal{L}_S) = 2\underline{\omega}(1 - \cos(2\pi/N)) > 0$.
\end{enumerate}
\begin{proof}
{Positive Semi-definiteness:} Symmetry follows from the undirected adjacency in $\mathcal{G}_N$. For any $x \in \mathbb{R}^N$, the quadratic form is $x^T \mathcal{L}_{\mathcal{G}} x = \sum_{k=1}^N \hat{\omega}_k (x_k - x_{k-1})^2 \ge 0$, where $x_0 \equiv x_N$. Then, $\mathcal{L}_{\mathcal{G}} \succeq 0$.
{Null Space:} $x^T \mathcal{L}_{\mathcal{G}} x = 0 \iff x_k = x_{k-1}$ for all $k$. Hence, $\text{ker}(\mathcal{L}_{\mathcal{G}}) = \text{span}(\mathbf{1}_N)$.
{Connectivity:} Since $\underline{\omega} > 0$, the cycle graph $\mathcal{G}_N$ is connected. By the Courant-Fischer Theorem~\cite[Th. 4.2.2]{horn}, the second smallest eigenvalue satisfies $\lambda_2(\mathcal{L}_{\mathcal{G}}) \ge \underline{\omega} \lambda_2(\mathcal{L}_S) > 0$, which establishes the exponential stability on the subspace $\mathbf{1}_N^{\perp}$.
\end{proof}

\section{Parameters Design}  \label{apc}
The centroid tracking gain $K^*$ and agent gain $k_p$ are tuned to ensure that agents navigate poly-annlus $S$ toward their respective centroids, while strictly adhering to the Riemannian metric constraints. The partition gain $k_{\hat{\psi}}$ is determined by the spectral properties of the communication topology to guarantee the exponential decay of workload imbalances across the multi-agents. To maintain topological connectivity in the presence of holes, the sequence factor $\beta$ is introduced, which ensures that the obstacle-avoiding partition bars remain disjoint and strictly separated from the boundary components in $\Xi$. Finally, the convergence threshold $T_\varepsilon$ is calibrated to define the steady-state precision of the coverage mission, ensuring the terminal configuration satisfies the quasi-optimal workload distribution.

\section{Algorithm Analysis} \label{app3}
This section evaluates the computational efficiency, scalability and robustbness of the proposed diffeomorphic coverage control framework.
\subsection{Computational Complexity}
The framework decouples offline mapping from online distributed execution to ensure real-time performance:
\begin{enumerate}
    \item {Mapping (Algorithm~\ref{Algorithm3}):} Solving the Laplace equation for the conformal map $\tau$ on $S$ with $v$ vertices across $M$ subdomains incurs a complexity of $O((V \cdot M^{-1})^{1.5})$.
    \item {Partition (Algorithm~\ref{Algorithm1}):} The linear intersection checks between $N$ partition bars and $n$ obstacles yield a complexity of $O(n)$ per control cycle.
    \item {Control (Algorithm~\ref{Algorithm2}):} Each iteration involves $O(D_{\mathcal{G}})$ communication rounds, where $D_{\mathcal{G}}$ is the graph diameter.
\end{enumerate}
The total online complexity is $O(n + K^* \cdot D_{\mathcal{G}})$. As shown in Table~\ref{table_metrics}, the average execution time $t_{avg} \approx 0.02$ s per agent, supporting high-frequency deployment.
\subsection{Scalability Analysis}
The poriferous geometry of $S$ imposes a theoretical limit on the agent population $N$ to maintain topological disjointness and ISS stability. Given a sequence factor $\beta$, the maximal agent capacity $N_{max}$ is determined by:
\begin{equation} \label{Nlimit}
N_{max} \le \min \{ 0.5 \hat{d}_{min} \cdot (\hat{r}_{max} \beta)^{-1}, \gamma\epsilon_{max} \cdot (C_{\hat{\delta}} \beta)^{-1} \},
\end{equation}
where $\hat{d}_{min}$ is the minimum obstacle separation, $\hat{r}_{max}$ is the maximum obstacle radius, and $\epsilon_{max}$ is the allowable workload error. The first term in \eqref{Nlimit} ensures the non-overlapping of buffer zones ($\hat{r}_k (1 + N \beta) - \hat{r}_k \le 0.5 \hat{d}_{min}$), while the second term guarantees the ultimate workload error $\| \mathbf{e}(\infty) \| \le C_{\hat{\delta}} \beta \cdot \gamma^{-1}$ remains within $\epsilon_{max}$, where $C_{\hat{\delta}}$ is the geometric sensitivity defined in Lemma~\ref{lemma 3.1}. This bound defines the fundamental trade-off between coverage precision $\beta$ and the swarm scale $N$ on poly annulus $S$.

\subsection{Robustness Analysis}  \label{af}
Consider a failure at $t = T_f$ where a subset of agents $N_{\mathcal{F}}$ ceases operation, resulting in an active set $\mathcal{A}$ with size $N_{\mathcal{A}} = N - N_{\mathcal{F}}$. Then, the new equilibrium workload is $\bar{m}_{\mathcal{A}} = M_{total} \cdot N_{\mathcal{A}}^{-1}$. Define the error vector $\mathbf{e} \in \mathbb{R}^{N_{\mathcal{A}}}$ with components $e_i = \hat{m}_i - \bar{m}_{\mathcal{A}}$, satisfying $\mathbf{e} \perp \mathbf{1}_{N_{\mathcal{A}}}$. Under the reduced topology, the error dynamics in $\Xi$ follow the Riemannian Laplacian flow $\dot{\mathbf{e}} = -k_{\hat{\psi}} \mathcal{L}_{\mathcal{G}}(\hat{\omega}) \mathbf{e}$, where $\mathcal{L}_{\mathcal{G}}(\hat{\omega})$ is the weighted Laplacian induced by the metric $\hat{\eta}$. For $V = \|\mathbf{e}\|^2/2$, its derivative satisfies
\begin{equation}
\dot{V} = -k_{\hat{\psi}} \mathbf{e}^T \mathcal{L}_{\mathcal{G}}(\hat{\omega}) \mathbf{e} \le -k_{\hat{\psi}} \cdot \underline{\omega} \cdot \lambda_2(\mathcal{L}_{S, N_{\mathcal{A}}}) \cdot |\mathbf{e}|^2,
\end{equation}
where $\underline{\omega} > 0$ is the minimum marginal density. The exponential convergence rate is $\gamma_{\mathcal{A}} = 2 k_{\hat{\psi}} \cdot \underline{\omega} \cdot ( 1 - \cos ( 2\pi \cdot N_{\mathcal{A}}^{-1} ) )$. Since $N_{\mathcal{A}} < N$, the spectral property of the cycle graph ensures $\gamma_{\mathcal{A}} > \gamma_{\mathcal{N}}$. The system's re-balancing speed automatically increases as the swarm size reduces, ensuring that the error $\|\mathbf{e}(t)\| \le \|\mathbf{e}(T_0)\| \exp( -\gamma_{\mathcal{A}} (t - T_f) )$ swiftly decays.  This intrinsic self-healing property allows the remaining agents to reach new Riemannian centroids without manual reconfiguration of the map $\tau$ or the Riemannian metric $\hat{\eta}$.

\end{document}